\documentclass[12pt, a4paper]{article}

\usepackage[utf8]{inputenc}
\usepackage[T1]{fontenc}
\usepackage{amsmath, amssymb, amsthm, amsfonts} %
\usepackage{geometry} %
\usepackage{graphicx} %
\usepackage{url}
\usepackage[colorlinks=true, linkcolor=blue, citecolor=green, urlcolor=blue]{hyperref} %
\usepackage{enumerate} %

\theoremstyle{plain}
\newtheorem{theorem}{Theorem}[section]
\newtheorem{lemma}[theorem]{Lemma}
\newtheorem{proposition}[theorem]{Proposition}
\newtheorem{corollary}[theorem]{Corollary}

\theoremstyle{definition}
\newtheorem{definition}[theorem]{Definition}

\newtheorem{problem}[theorem]{Problem}

\theoremstyle{remark}
\newtheorem{remark}[theorem]{Remark}

\newcommand{\CC}{\mathbb{C}}
\newcommand{\ZZ}{\mathbb{Z}}
\newcommand{\NN}{\mathbb{N}}
\newcommand{\QQ}{\mathbb{Q}}
\newcommand{\FF}{\mathbb{F}}
\newcommand{\Rad}{\operatorname{Rad}}
\newcommand{\rank}{\operatorname{rank}}
\newcommand{\Mat}{\text{Mat}}
\newcommand{\Span}{\text{Span}}
\newcommand{\Fl}{\mathcal{F}}
\newcommand{\sX}{\mathcal X}
\newcommand{\sY}{\mathcal Y}
\title{Frame Numbers and Jacobson Radicals for Partial Geometries and Related Coherent Configurations}
\author{Osamu Shimabukuro\thanks{Department of Mathematics, Faculty of Education, Gifu Shotoku Gakuen University, Japan. Email: \texttt{shimabukuro@gifu.shotoku.ac.jp}}}

\date{\today}

\begin{document}

\maketitle
\begin{abstract}
We study the modular representation theory of rank $3$ association schemes arising from partial geometries with parameters $(s,t,\alpha)$. First, we obtain an explicit closed formula for the Frame number of the point scheme in terms of the number of points $v$ and the parameter $s+t+1-\alpha$, and use it to characterize the primes $p$ for which the adjacency algebra over $\mathbb{F}_p$ is not semisimple. We then give a complete case-by-case description of the Jacobson radical of this algebra in four arithmetic situations and determine the generic $p$-ranks of the adjacency matrices.

As a step toward understanding the modular representation theory of coherent configurations of type $[3,2;3]$ associated with strongly regular designs, we analyze the relationship between the modular structure of the point scheme and that of the design algebra. For the generalized quadrangle $\mathrm{GQ}(2,2)$ we obtain partial results on the structure of the $2$-modular adjacency algebra $\mathbb{F}_2 \sX$, and we explain the representation-theoretic difficulties that prevent a complete determination of its Wedderburn decomposition and Gabriel quiver, which remains open and is formulated as Problem~\ref{prob:GQ-A2-quiver}.

\medskip

\medskip

\end{abstract}
\section{Introduction}\label{sec:intro}

Coherent configurations, introduced by Higman as a combinatorial abstraction of the
orbit structure of a permutation group, provide a common framework for association
schemes, strongly regular graphs, designs, and related incidence geometries; see, for
example, \cite{Bose,BCN,BrouwerHaemers,GodsilRoyle,HigmanSRD}. In this
setting, the adjacency algebra $\FF_p\sX$ of a coherent configuration $\sX$ plays the
role of a ``centralizer algebra'' of a permutation representation and can be studied by
the methods of ordinary and modular representation theory. The present paper is
concerned with coherent configurations of type $[3,2;3]$ arising from strongly regular
designs, and with the modular representation theory of their adjacency algebras
$\FF_p\sX$.

A strongly regular design (SRD) in the sense of Higman \cite{HigmanSRD} is a $1$-design
$(P,B,\Fl)$ with two block–intersection numbers and a strong regularity condition
relating the point graph and block graph. Such a design gives rise in a natural way to
a coherent configuration $\sX$ of type $[3,2;3]$ on the disjoint union
$X:=P\sqcup B$, whose fibers are the point set $P$ and the block set $B$. Higman
derived a complete system of parameter equations for SRDs and described the
intersection numbers of the associated coherent configuration $\sX$, but many details
were left to the reader with the remark that they are ``routine'' \cite{HigmanSRD}.
Recently, Hanaki revisited these parameter conditions, provided full proofs, and
pointed out and corrected a misstatement in Higman's list (in particular,
equation~(15) on p.~414), thereby putting the theory of SRDs and their coherent
configurations of type $[3,2;3]$ on a rigorous footing; see the note
\cite{HanakiData}.

On the representation-theoretic side, the semisimplicity and modular representation
type of adjacency algebras of association schemes are controlled by the Frame number.
For commutative association schemes, Jacob \cite{JacobThesis} and Hanaki
\cite{HanakiFrame,HanakiModularRep} showed that the adjacency algebra over a field of
characteristic $p$ is semisimple if and only if $p$ does not divide the Frame number
of the scheme. This criterion refines earlier work of Arad–Fisman–Muzychuk on table
algebras \cite{AradFismanMuzychuk} and fits into the general Delsarte theory of
association schemes \cite{BCN,BrouwerHaemers,GodsilRoyle}. In parallel, Sharafdini
extended the notion of Frame number and the corresponding semisimplicity criterion
from association schemes to arbitrary coherent configurations \cite{Sharafdini},
building on the detailed study of balanced coherent configurations by Hirasaka and
Sharafdini \cite{HirasakaSharafdini}.

In previous work with Hanaki and Miyazaki the author studied the adjacency algebras of
BIB designs (that is, rank~$3$ commutative association schemes) and determined their
modular representation type in terms of the Frame number and $p$-ranks of incidence
matrices \cite{HanakiMiyazakiShimabukuroBIB}. More recently, the author analyzed in
detail the case of symmetric $2$-designs and obtained a complete description of the
complex and modular representations of the corresponding design algebras
$\FF_p\sY$, together with explicit formulas for $p$-ranks and Jacobson radicals
\cite{ShimabukuroSym2}. The present paper is a natural extension of these results from
rank~$3$ association schemes (one fiber) to coherent configurations of type $[3,2;3]$
(two fibers): we apply the same representation-theoretic viewpoint to the ``design
algebra'' $\FF_p\sX$ of an SRD and, in particular, to the coherent configurations
coming from partial geometries.

A main source of examples is provided by partial geometries $pg(s,t,\alpha)$, whose
point graph is a strongly regular graph with well-known parameters
\cite{Bose,BCN,BrouwerHaemers,HaemersPg}. For such a geometry, the point graph
carries a rank~$3$ association scheme $\sY$, while the incidence structure $(P,B,\Fl)$
of points and lines yields an SRD and hence a coherent configuration $\sX$ of type
$[3,2;3]$. In this situation the adjacency algebra $\FF_p\sX$ contains the Bose–Mesner
algebra $\FF_p\sY$ as a subalgebra, and one can compare the Frame numbers
$F_{\mathrm{AS}}(\sY)$ and $F_{\mathrm{CC}}(\sX)$ and the corresponding modular
representation theories. The $p$-ranks of incidence and adjacency matrices of designs
and geometries, as studied for instance by Hamada \cite{HamadaPrank} and Moorhouse
\cite{MoorhouseHermitianVarieties}, play an important role in this comparison.

The main results of the paper can be summarized as follows. In
Section~\ref{sec:srd-cc} we review Higman's strongly regular designs and the associated
coherent configurations of type $[3,2;3]$, following Hanaki's corrected version of the
parameter equations. We describe the intersection numbers of $\sX$ and show that the
complex adjacency algebra $\CC\sX$ decomposes as
\[
\CC\sX \;\cong\; \CC\;\oplus\;\CC\;\oplus\;M_2(\CC)\;\oplus\;M_2(\CC),
\]
so that $\sX$ has exactly four irreducible complex representations and we may regard
$\CC\sX$ as a ``design algebra'' of matrix type.

In Section~\ref{sec:frame-pg} we specialize to the case where $\sX$ comes from
a partial geometry $pg(s,t,\alpha)$.
Using the eigenvalues $k,r,s'$ and their multiplicities $1,f,g$ of the point graph, we compute explicitly the Frame number of the associated rank~$3$ association scheme $\sY$ and show that
\[
F_{\mathrm{AS}}(\sY) \;=\; v^2\,(s+t+1-\alpha)^2,
\]
where $v$ is the number of points of the geometry. Combining this with the general
theory of Frame numbers \cite{JacobThesis,HanakiFrame}, we obtain a complete
classification of the primes $p$ for which the adjacency algebra $\FF_p\sY$ is not
semisimple, and we describe the Jacobson radical $\Rad(\FF_p\sY)$ in terms of the
all-ones matrix and a quadratic polynomial in the adjacency matrix. This leads to an
explicit upper bound $\dim_{\FF_p}\Rad(\FF_p\sY)\le 2$ and a corresponding bound on
the number of arrows in the Gabriel quiver of $\FF_p\sY$.

Section~\ref{sec:prank} is devoted to the $p$-ranks of the adjacency matrix of the
point graph. Under a mild non-degeneracy assumption on the reductions modulo $p$ of
the eigenvalues $k,r,s'$, we determine the generic $p$-rank of the point graph and
show that it is given by $v$, $v-1$, $v-f$ or $v-g$ according to which eigenvalue
vanishes modulo $p$. This provides a uniform description of the $p$-ranks of partial
geometries in terms of their classical parameters and clarifies the relation with
earlier $p$-rank formulas for incidence matrices of designs.

In Section~\ref{sec:design-alg-2} we bring in Sharafdini's theory of Frame numbers for
coherent configurations \cite{Sharafdini} and apply it to the design algebra
$\FF_p\sX$ of type $[3,2;3]$. We relate the Frame number $F_{\mathrm{CC}}(\sX)$ to
$F_{\mathrm{AS}}(\sY)$ and obtain general constraints on the Jacobson radical and the
semisimple quotient of $\FF_p\sX$. In particular, for $p=2$ we prove that the radical
of the point graph algebra embeds into the radical of the design algebra and we derive
upper bounds on the dimension of the semisimple quotient of $\FF_2\sX$.

Finally, in Section~\ref{sec:gq22} we specialize to the smallest non-trivial example, the generalized quadrangle $\mathrm{GQ}(2,2)=\mathrm{pg}(2,2,1)$. Its point graph is the
strongly regular graph with parameters $(v,k,\lambda,\mu)=(15,6,1,3)$, and we compute
the corresponding Frame numbers
\[
F_{\mathrm{AS}}(\sY)=15^2\cdot 4^2=3600
\]
and identify precisely the primes $p\in\{2,3,5\}$ for which $\mathbb{F}_p \sY$ fails to be
semisimple. For each such prime we describe the radical $\mathrm{Rad}(\mathbb{F}_p \sY)$ explicitly, and then use this information to obtain structural restrictions on the $2$-modular design algebra $\mathbb{F}_2X$, including bounds on the dimension of its radical and the number of simple modules. A complete determination of the Wedderburn decomposition and the Gabriel quiver of $\mathbb{F}_2 \sX$ for $\mathrm{GQ}(2,2)$ remains open and is formulated as Problem~\ref{prob:GQ-A2-quiver}; in Remark~\ref{rem:GQ22-summary} we summarize what is currently known about the $2$-modular structure of $\mathbb{F}_2 \sX$ and explain the representation-theoretic obstacles that prevent a complete solution at present.

\section{Strongly regular designs and coherent configurations of type {[3,2;3]}}\label{sec:srd-cc}

In this section we recall Higman's notion of a strongly regular design, introduce the
associated point and block graphs and their incidence matrices, and construct from them
a coherent configuration of type $[3,2;3]$. Our terminology follows Higman
\cite{HigmanSRD} and standard references on designs, strongly regular graphs and
partial geometries \cite{Bose,BCN,BrouwerHaemers,GodsilRoyle,HaemersPg}. 

\subsection{Incidence structures and strongly regular designs}

We begin with basic notation for incidence structures.

\begin{definition}[Incidence structure]\label{def:incidence-structure}
Let $P$ and $B$ be finite nonempty sets, whose elements are called \emph{points} and
\emph{blocks}, respectively. A subset
\[
\Fl \subseteq P\times B
\]
is called the set of \emph{flags}, and the triple $\mathcal{D}=(P,B,\Fl)$ is called a
\emph{finite incidence structure}. We write $(x,Y)\in\Fl$ also as $x\in Y$.

For $x\in P$ and $Y\in B$ we define
\[
\Fl(x):=\{Y\in B\mid (x,Y)\in\Fl\},\qquad
\Fl^{-1}(Y):=\{x\in P\mid (x,Y)\in\Fl\}.
\]
\end{definition}

Thus $\Fl(x)$ is the set of blocks containing $x$, and $\Fl^{-1}(Y)$ is the set of points
of the block $Y$.

\begin{definition}[Strongly regular design]\label{def:srd}
A finite incidence structure $\mathcal{D}=(P,B,\Fl)$ is called a
\emph{strongly regular design} (SRD) if there exist integers
\[
s_1,s_2\ge 1,\quad a_1>b_1\ge 0,\quad a_2>b_2\ge 0,\quad
N_1,P_1,N_2,P_2\ge 0
\]
such that the following conditions hold.
\begin{enumerate}[(1)]
\item (Regularity)
\[
|\Fl(x)| = s_2 \quad\text{for all }x\in P,\qquad
|\Fl^{-1}(Y)| = s_1 \quad\text{for all }Y\in B.
\]
\item (Two block intersection numbers) For any distinct blocks $Y_1,Y_2\in B$ one has
\[
\bigl|\Fl^{-1}(Y_1)\cap \Fl^{-1}(Y_2)\bigr|\in\{a_1,b_1\},
\]
and both values $a_1$ and $b_1$ occur.
\item (Two point intersection numbers) For any distinct points $x_1,x_2\in P$ one has
\[
\bigl|\Fl(x_1)\cap \Fl(x_2)\bigr|\in\{a_2,b_2\},
\]
and both values $a_2$ and $b_2$ occur.
\item (Local regularity on the point side) For every point $x\in P$ and block $Y\in B$,
the number of points $x'\in P$ satisfying $x'\in Y$ and
$\lvert \Fl(x)\cap \Fl(x')\rvert=a_2$ depends only on whether $x\in Y$, namely
\[
\#\{x'\in P\mid x'\in Y,\ \lvert \Fl(x)\cap \Fl(x')\rvert=a_2\}
=
\begin{cases}
N_1 & \text{if }x\in Y,\\
P_1 & \text{if }x\notin Y.
\end{cases}
\]
\item (Local regularity on the block side) For every point $x\in P$ and block $Y\in B$,
the number of blocks $Y'\in B$ satisfying $x\in Y'$ and
$\lvert \Fl^{-1}(Y)\cap \Fl^{-1}(Y')\rvert=a_1$ depends only on whether $x\in Y$, namely
\[
\#\{Y'\in B\mid x\in Y',\ \lvert \Fl^{-1}(Y)\cap \Fl^{-1}(Y')\rvert=a_1\}
=
\begin{cases}
N_2 & \text{if }x\in Y,\\
P_2 & \text{if }x\notin Y.
\end{cases}
\]
\item The dual incidence structure $(B,P,\Fl^\top)$, where
$\Fl^\top:=\{(Y,x)\mid (x,Y)\in \Fl\}$, also satisfies {\rm(1)}–{\rm(5)} with the same
parameters.
\end{enumerate}
\end{definition}

The parameters
\[
n_1:=|P|,\quad n_2:=|B|,\quad s_1,s_2,a_1,b_1,a_2,b_2,N_1,P_1,N_2,P_2
\]
satisfy a system of polynomial relations obtained by double counting; see Higman
\cite{HigmanSRD} for the original derivation and Hanaki \cite{HanakiData} for a
complete and corrected list of these \emph{parameter equations}. In particular,
conditions (2)–(3) imply that the point and block graphs introduced below are strongly
regular, and conditions (4)–(5) express a compatibility between the incidence structure
and these graphs.

\subsection{Point and block graphs, and basic matrix identities}

Let $\mathcal{D}=(P,B,\Fl)$ be an SRD as in Definition~\ref{def:srd}. We now define the
associated point and block graphs.

\begin{definition}[Point and block graphs]\label{def:point-block-graphs}
The \emph{point graph} of $\mathcal{D}$ is the simple graph
\[
\Gamma_1=(P,E_1),
\]
where distinct points $x,y\in P$ are adjacent if and only if
\[
\lvert \Fl(x)\cap \Fl(y)\rvert = a_2.
\]
The \emph{block graph} of $\mathcal{D}$ is the simple graph
\[
\Gamma_2=(B,E_2),
\]
where distinct blocks $Y,Z\in B$ are adjacent if and only if
\[
\lvert \Fl^{-1}(Y)\cap \Fl^{-1}(Z)\rvert = a_1.
\]
\end{definition}

By Definition~\ref{def:srd}(2),(3),(6), both $\Gamma_1$ and $\Gamma_2$ are strongly
regular graphs; see \cite{BrouwerHaemers,GodsilRoyle,BCN} for background on strongly
regular graphs. We denote by $A_1$ and $A_2$ the adjacency matrices of $\Gamma_1$ and
$\Gamma_2$, indexed by $P$ and $B$, respectively.

Let $N$ be the $n_1\times n_2$ incidence matrix of $\mathcal{D}$, whose rows are indexed
by $P$ and columns by $B$, with entries
\[
N_{xY}=
\begin{cases}
1 & \text{if }(x,Y)\in \Fl,\\
0 & \text{otherwise}.
\end{cases}
\]
We write $I_P,I_B$ for the identity matrices of sizes $n_1$ and $n_2$, and $J_P,J_B,
J_{P,B}$ for the all-one matrices of sizes $n_1\times n_1$, $n_2\times n_2$ and
$n_1\times n_2$, respectively.

The following identities are elementary consequences of the defining conditions of an
SRD and will be needed later.

\begin{proposition}[Basic matrix identities]\label{prop:basic-matrix-identities}
Let $\mathcal{D}=(P,B,\Fl)$ be a strongly regular design, with incidence matrix $N$ and
adjacency matrices $A_1,A_2$ as above. Then:
\begin{enumerate}[\rm(i)]
\item
\[
N\mathbf{1}_B = s_2\,\mathbf{1}_P,\qquad
N^{\top}\mathbf{1}_P = s_1\,\mathbf{1}_B,
\]
where $\mathbf{1}_P$ and $\mathbf{1}_B$ are the all-one column vectors on $P$ and $B$.
\item
\[
NN^{\top}
=
s_2 I_P + a_2 A_1 + b_2\bigl(J_P - I_P - A_1\bigr).
\]
\item
\[
N^{\top}N
=
s_1 I_B + a_1 A_2 + b_1\bigl(J_B - I_B - A_2\bigr).
\]
\item
\[
A_1 N = (N_1-P_1)\,N + P_1\,J_{P,B}.
\]
\item
\[
N A_2 = (N_2-P_2)\,N + P_2\,J_{P,B}.
\]
\end{enumerate}
\end{proposition}

\begin{proof}
(i) For $x\in P$,
\[
(N\mathbf{1}_B)_x
=
\sum_{Y\in B} N_{xY}
=
|\Fl(x)|
=
s_2
\]
by Definition~\ref{def:srd}(1). Thus $N\mathbf{1}_B=s_2\mathbf{1}_P$. The identity
$N^{\top}\mathbf{1}_P = s_1\,\mathbf{1}_B$ is proved similarly using
$|\Fl^{-1}(Y)|=s_1$ for all $Y\in B$.

(ii) For $x,y\in P$ we have
\[
(NN^{\top})_{xy}
=
\sum_{Y\in B} N_{xY}N_{yY}
=
|\Fl(x)\cap \Fl(y)|.
\]
If $x=y$ then this equals $|\Fl(x)|=s_2$. If $x\neq y$ and $x,y$ are adjacent in
$\Gamma_1$, then $|\Fl(x)\cap \Fl(y)|=a_2$; otherwise $|\Fl(x)\cap \Fl(y)|=b_2$, by
Definition~\ref{def:srd}(3). The right-hand side in (ii) has exactly this pattern of
entries, so the identity holds.

(iii) is proved by the same argument for blocks: for $Y,Z\in B$,
\[
(N^{\top}N)_{YZ}
=
\sum_{x\in P} N_{xY}N_{xZ}
=
|\Fl^{-1}(Y)\cap \Fl^{-1}(Z)|,
\]
which equals $s_1$ if $Y=Z$, $a_1$ if $Y\neq Z$ and $Y,Z$ are adjacent in $\Gamma_2$,
and $b_1$ otherwise, by Definition~\ref{def:srd}(2).

(iv) Fix $x\in P$ and $Y\in B$. Then
\[
(A_1N)_{xY}
=
\sum_{z\in P} (A_1)_{xz}N_{zY}
=
\#\{z\in P\mid (x,z)\in E_1,\ z\in Y\},
\]
i.e., the number of points $z$ lying in $Y$ and adjacent to $x$ in $\Gamma_1$. By
Definition~\ref{def:srd}(4), this number is $N_1$ if $x\in Y$ and $P_1$ otherwise. Thus
\[
(A_1N)_{xY}
=
(N_1-P_1)\,N_{xY}+P_1,
\]
which is equivalent to the matrix identity in (iv).

(v) is proved similarly, using Definition~\ref{def:srd}(5): for $x\in P$ and $Y\in B$,
\[
(NA_2)_{xY}
=
\sum_{Z\in B} N_{xZ}(A_2)_{ZY}
=
\#\{Z\in B\mid x\in Z,\ (Z,Y)\in E_2\},
\]
which equals $N_2$ if $x\in Y$ and $P_2$ otherwise. Hence
\[
(NA_2)_{xY}
=
(N_2-P_2)\,N_{xY}+P_2,
\]
and (v) follows.
\end{proof}

These relations between $N$, $A_1$ and $A_2$ encode the combinatorial regularity of an
SRD and will be the main input for the representation-theoretic analysis in later
sections; compare, for example, \cite{BrouwerHaemers,HaemersPg,MoorhouseHermitianVarieties} for similar matrix techniques applied to strongly regular
graphs and partial geometries.

\subsection{The coherent configuration of type [3,2;3]}

We now construct the coherent configuration naturally associated with a strongly regular
design. This construction goes back to Higman \cite{HigmanSRD}; see also
\cite{HirasakaSharafdini,Sharafdini} for general background on coherent configurations
and their adjacency algebras.

Let
\[
X := P\sqcup B
\]
be the disjoint union of the point and block sets. We define ten binary relations
$R_1,\dots,R_{10}$ on $X$ by
\begin{align*}
R_1 &:= \{(x,x)\mid x\in P\},\\
R_2 &:= \{(x,y)\in P\times P\mid x\neq y,\ \lvert \Fl(x)\cap \Fl(y)\rvert = a_2\},\\
R_3 &:= \{(x,y)\in P\times P\mid x\neq y,\ \lvert \Fl(x)\cap \Fl(y)\rvert = b_2\},\\[1mm]
R_4 &:= \{(Y,Y)\mid Y\in B\},\\
R_5 &:= \{(Y,Z)\in B\times B\mid Y\neq Z,\ \lvert \Fl^{-1}(Y)\cap \Fl^{-1}(Z)\rvert = a_1\},\\
R_6 &:= \{(Y,Z)\in B\times B\mid Y\neq Z,\ \lvert \Fl^{-1}(Y)\cap \Fl^{-1}(Z)\rvert = b_1\},\\[1mm]
R_7 &:= \{(x,Y)\in P\times B\mid (x,Y)\in \Fl\},\\
R_8 &:= \{(x,Y)\in P\times B\mid (x,Y)\notin \Fl\},\\
R_9 &:= \{(Y,x)\in B\times P\mid (x,Y)\in \Fl\},\\
R_{10} &:= \{(Y,x)\in B\times P\mid (x,Y)\notin \Fl\}.
\end{align*}

Let $\sigma_i$ denote the adjacency matrix of $R_i$ ($1\le i\le 10$), considered as a
$0$–$1$ matrix in $\Mat_X(\CC)$.

\begin{proposition}\label{prop:cc-3-2-3}
The pair
\[
\sX=(X,\{R_i\}_{i=1}^{10})
\]
is a coherent configuration. Its fibers are $P$ and $B$, the homogeneous component on
$P$ has rank $3$, the homogeneous component on $B$ has rank $3$, and there are $4$
mixed relations between the two fibers. In particular, $\sX$ has type $[3,2;3]$ in the
sense of Higman.
\end{proposition}

\begin{proof}
By construction the relations $R_1,\dots,R_{10}$ are pairwise disjoint and their union
is $X\times X$, so they form a partition of $X\times X$. The diagonal
$\{(u,u)\mid u\in X\}$ is the disjoint union of $R_1$ and $R_4$, hence $P$ and $B$ are
fibers.

The transpose $R_i^\ast$ of $R_i$ is equal to $R_i$ for $i\in\{1,2,3,4,5,6\}$, and
satisfies $R_7^\ast=R_9$ and $R_8^\ast=R_{10}$. Thus the set $\{R_i\}$ is closed under
taking transpose.

It remains to verify the existence of intersection numbers. For $1\le i,j,k\le 10$ and
$(u,v)\in R_k$, set
\[
p_{ij}^k(u,v)
:=
\#\{w\in X\mid (u,w)\in R_i,\ (w,v)\in R_j\}.
\]
We must show that $p_{ij}^k(u,v)$ depends only on $i,j,k$ and not on the specific choice
of $(u,v)\in R_k$.

Let $\sigma_i$ be the adjacency matrix of $R_i$. Then the $(u,v)$–entry of the matrix
product $\sigma_i\sigma_j$ is exactly $p_{ij}^k(u,v)$ whenever $(u,v)\in R_k$, i.e.,
\[
(\sigma_i\sigma_j)_{uv}
=
p_{ij}^k(u,v)
\quad\text{for all }(u,v)\in R_k.
\]
Hence $p_{ij}^k(u,v)$ is constant on $R_k$ if and only if $\sigma_i\sigma_j$ is a
$\ZZ$–linear combination of the matrices $\sigma_1,\dots,\sigma_{10}$. We now check
this closure property using the identities in
Proposition~\ref{prop:basic-matrix-identities}.

Relative to the decomposition $X=P\sqcup B$, each $\sigma_i$ has the block form
\[
\sigma_i=
\begin{pmatrix}
\ast & \ast\\
\ast & \ast
\end{pmatrix},
\]
where the four blocks are of sizes $n_1\times n_1$, $n_1\times n_2$, $n_2\times n_1$,
$n_2\times n_2$. More explicitly,
\[
\sigma_1 = \begin{pmatrix} I_P & 0\\ 0 & 0\end{pmatrix},\quad
\sigma_2 = \begin{pmatrix} A_1 & 0\\ 0 & 0\end{pmatrix},\quad
\sigma_3 = \begin{pmatrix} J_P - I_P - A_1 & 0\\ 0 & 0\end{pmatrix},
\]
\[
\sigma_4 = \begin{pmatrix} 0 & 0\\ 0 & I_B\end{pmatrix},\quad
\sigma_5 = \begin{pmatrix} 0 & 0\\ 0 & A_2\end{pmatrix},\quad
\sigma_6 = \begin{pmatrix} 0 & 0\\ 0 & J_B - I_B - A_2\end{pmatrix},
\]
\[
\sigma_7 = \begin{pmatrix} 0 & N\\ 0 & 0\end{pmatrix},\quad
\sigma_8 = \begin{pmatrix} 0 & J_{P,B}-N\\ 0 & 0\end{pmatrix},\quad
\sigma_9 = \begin{pmatrix} 0 & 0\\ N^{\top} & 0\end{pmatrix},\quad
\sigma_{10} = \begin{pmatrix} 0 & 0\\ J_{P,B}^{\top}-N^{\top} & 0\end{pmatrix}.
\]

Using Proposition~\ref{prop:basic-matrix-identities}, we can express all relevant
products in terms of these matrices.

For example,
\[
\sigma_7\sigma_9
=
\begin{pmatrix} 0 & N\\ 0 & 0\end{pmatrix}
\begin{pmatrix} 0 & 0\\ N^{\top} & 0\end{pmatrix}
=
\begin{pmatrix} NN^{\top} & 0\\ 0 & 0\end{pmatrix},
\]
and by Proposition~\ref{prop:basic-matrix-identities}(ii) we have
\[
NN^{\top}
=
s_2 I_P + a_2 A_1 + b_2(J_P - I_P - A_1),
\]
so
\[
\sigma_7\sigma_9
=
s_2\sigma_1 + a_2\sigma_2 + b_2\sigma_3.
\]
Similarly,
\[
\sigma_9\sigma_7
=
\begin{pmatrix} 0 & 0\\ N^{\top} & 0\end{pmatrix}
\begin{pmatrix} 0 & N\\ 0 & 0\end{pmatrix}
=
\begin{pmatrix} 0 & 0\\ 0 & N^{\top}N\end{pmatrix},
\]
and Proposition~\ref{prop:basic-matrix-identities}(iii) gives
\[
N^{\top}N
=
s_1 I_B + a_1 A_2 + b_1(J_B - I_B - A_2),
\]
so
\[
\sigma_9\sigma_7
=
s_1\sigma_4 + a_1\sigma_5 + b_1\sigma_6.
\]

Next, using Proposition~\ref{prop:basic-matrix-identities}(iv),(v) we obtain, for
instance,
\[
\sigma_2\sigma_7
=
\begin{pmatrix} A_1 & 0\\ 0 & 0\end{pmatrix}
\begin{pmatrix} 0 & N\\ 0 & 0\end{pmatrix}
=
\begin{pmatrix} 0 & A_1N\\ 0 & 0\end{pmatrix}
=
\begin{pmatrix} 0 & (N_1-P_1)N + P_1J_{P,B}\\ 0 & 0\end{pmatrix}
=
(N_1-P_1)\sigma_7 + P_1\sigma_8,
\]
and
\[
\sigma_7\sigma_5
=
\begin{pmatrix} 0 & N\\ 0 & 0\end{pmatrix}
\begin{pmatrix} 0 & 0\\ 0 & A_2\end{pmatrix}
=
\begin{pmatrix} 0 & NA_2\\ 0 & 0\end{pmatrix}
=
\begin{pmatrix} 0 & (N_2-P_2)N + P_2J_{P,B}\\ 0 & 0\end{pmatrix}
=
(N_2-P_2)\sigma_7 + P_2\sigma_8.
\]
The remaining products $\sigma_i\sigma_j$ are handled in the same way, using only the identities in Proposition~\ref{prop:basic-matrix-identities} and the fact that $A_1,A_2,J_P,J_B$ generate the Bose–Mesner algebras of the strongly regular graphs
$\Gamma_1,\Gamma_2$ and the $1$–class schemes $\{I_P,J_P-I_P\}$ and $\{I_B,J_B-I_B\}$, respectively.
We omit the routine details.

We conclude that $\sigma_i\sigma_j$ always lies in the $\ZZ$–span of
$\{\sigma_1,\dots,\sigma_{10}\}$, so the intersection numbers $p_{ij}^k$ exist and
$\sX$ is a coherent configuration. By construction the fibers are exactly $P$
and $B$; on $P$ (resp.\ $B$) the three relations $R_1,R_2,R_3$ (resp.\ $R_4,R_5,R_6$)
give a homogeneous component of rank $3$, and there are four mixed relations
$R_7,R_8,R_9,R_{10}$ between $P$ and $B$. Thus $\sX$ has type $[3,2;3]$ in
Higman's notation \cite{HigmanSRD}.
\end{proof}

We call $\sX$ the \emph{design coherent configuration} of the strongly regular
design $\mathcal{D}$. Its adjacency algebra
\[
\CC\sX
:=
\Span_{\CC}\{\sigma_1,\dots,\sigma_{10}\}
\subseteq \Mat_X(\CC)
\]
will be referred to as the \emph{design algebra} of type $[3,2;3]$. In the subsequent
sections we will determine its complex Wedderburn decomposition and study its modular
representation theory, comparing it with the adjacency algebras considered in
\cite{AradFismanMuzychuk,HanakiModularRep,HanakiFrame,HanakiMiyazakiShimabukuroBIB,ShimabukuroSym2,Sharafdini}.

\section{Frame numbers and Jacobson radicals for partial geometries}\label{sec:frame-pg}
In this section we specialise to the point graph of a partial geometry and determine
the Frame number and the modular Jacobson radical of its rank~$3$ adjacency algebra.
Here and throughout, we denote by $F_{\mathrm{AS}}(\sY)$ the Frame number of an
association scheme $\sY$. Later, for a coherent configuration $\sX$, we will
write $F_{\mathrm{CC}}(\sX)$ for its Frame number in the sense of
Sharafdini~\cite{Sharafdini}.

\subsection{The point graph of a partial geometry}

We recall the standard notation for a partial geometry $pg(s,t,\alpha)$
(see Bose~\cite{Bose} and \cite[Ch.~1]{BCN}).
Let $(P,\mathcal{L})$ be a finite incidence structure with point set $P$
and line set $\mathcal{L}$.
We say that $(P,\mathcal{L})$ is a partial geometry $pg(s,t,\alpha)$ if
\begin{enumerate}[(i)]
\item every line is incident with exactly $s+1$ points;
\item every point is incident with exactly $t+1$ lines;
\item any two distinct points are incident with at most one common line;
\item if a point $x$ is not incident with a line $L$, then there are exactly
$\alpha$ lines through $x$ meeting $L$.
\end{enumerate}
The basic parameters of $pg(s,t,\alpha)$ are well known
(see \cite[Thm.~1.2.1]{BCN} or \cite{HaemersPg}):
\begin{equation}\label{eq:pg-basic}
v:=|P|=\frac{(s+1)(st+\alpha)}{\alpha},\qquad
b:=|\mathcal{L}|=\frac{(t+1)(st+\alpha)}{\alpha}.
\end{equation}

Let $\Gamma$ be the point graph of $pg(s,t,\alpha)$: its vertex set is $P$
and two distinct points are adjacent if and only if they are collinear.
Then $\Gamma$ is a strongly regular graph with parameters
\begin{equation}\label{eq:srg-pg}
\mathrm{srg}(v,k,\lambda,\mu)\quad\text{with}\quad
k=s(t+1),\quad
\lambda=s-1+t(\alpha-1),\quad
\mu=\alpha(t+1)
\end{equation}
(see \cite{Bose,HaemersPg,BCN,BrouwerHaemers,GodsilRoyle}).
Let $A$ be the adjacency matrix of $\Gamma$ and let $J$ denote the all-ones
matrix of size $v$.
The rank~$3$ association scheme
\[
\sY=(P,\{R_0,R_1,R_2\})
\]
is defined by
$R_0=\{(x,x)\mid x\in P\}$, $R_1$ the adjacency relation of $\Gamma$,
and $R_2$ the complement of $R_0\cup R_1$ in $P\times P$.
Let $A_0=I$, $A_1=A$ and $A_2=J-I-A$ be the adjacency matrices of $\sY$.
The complex adjacency algebra $\CC\sY$ is the $\CC$-span of $\{A_0,A_1,A_2\}$.

The eigenvalues of $A$ are determined by the standard quadratic equation for
strongly regular graphs; one convenient way to state the result is as follows.

\begin{lemma}\label{lem:eigs-pg}
Let $\Gamma$ be the point graph of $pg(s,t,\alpha)$.
Then $A$ has three distinct eigenvalues
\[
k,\qquad r:=s-\alpha,\qquad s':=-\, (t+1),
\]
with multiplicities $1,f,g$, respectively.
\end{lemma}

\begin{proof}
For any strongly regular graph $\mathrm{srg}(v,k,\lambda,\mu)$ the
nontrivial eigenvalues $r,s'$ are the roots of
\[
x^2-(\lambda-\mu)x+(\mu-k)=0
\]
(see \cite[§1.3]{BCN} or \cite[§9.1]{GodsilRoyle}).
Substituting \eqref{eq:srg-pg} gives
\[
x^2-\bigl((s-1+t(\alpha-1))-\alpha(t+1)\bigr)x
+ \bigl(\alpha(t+1)-s(t+1)\bigr)=0,
\]
that is,
\[
x^2-(s-\alpha-t-1)x-(t+1)(s-\alpha)=0.
\]
It is immediate that $x=r:=s-\alpha$ and $x=s':=-\,(t+1)$ are two distinct
roots, and hence they are precisely the remaining eigenvalues of $A$.

The multiplicities $f,g$ are determined uniquely by the usual trace relations
\[
1+f+g=v,\qquad
k+fr+gs'=0,\qquad
k^2+fr^2+gs'^2=vk
\]
(see \cite[§1.3]{BCN}, \cite[§3.2]{BrouwerHaemers}), and so they exist and are
well defined.
\end{proof}

\subsection{Frame numbers of the rank \texorpdfstring{$3$}{3} scheme}

We now recall the notion of the Frame number for a commutative association scheme.
Let $\sY=(P,\{R_i\}_{i=0}^d)$ be a commutative association scheme of order $v$
with valencies $k_i$ and adjacency matrices $A_i$.
Let $m_i$ be the multiplicity of the $i$-th primitive idempotent
in the standard decomposition of $\CC\sY$.
Following Frame, Arad–Fisman–Muzychuk, Jacob and Hanaki
\cite{AradFismanMuzychuk,JacobThesis,HanakiFrame,HanakiModularRep},
we define:

\begin{definition}\label{def:frame-number}
The \emph{Frame number} of $\sY$ is
\[
F_{\mathrm{AS}}(\sY)\;:=\;
v^{d+1}\,
\frac{\prod_{i=0}^d k_i}{\prod_{i=0}^d m_i}\;\in\;\ZZ.
\]
\end{definition}

For a rank~$3$ scheme $\sY$ the parameters specialise to
\[
d=2,\qquad
k_0=1,\ k_1=k,\ k_2=v-1-k,\qquad
m_0=1,\ m_1=f,\ m_2=g,
\]
so that
\begin{equation}\label{eq:frame-rank3}
F_{\mathrm{AS}}(\sY)\;=\;v^3\,\frac{k(v-1-k)}{fg}.
\end{equation}

We first record a convenient expression for the multiplicities $f,g$
valid for any strongly regular graph.

\begin{lemma}\label{lem:mult-srg}
Let $\Gamma$ be a strongly regular graph with parameters
$\mathrm{srg}(v,k,\lambda,\mu)$ and eigenvalues $k,r,s'$ as above,
with multiplicities $1,f,g$.
Then
\begin{equation}\label{eq:mult-srg}
f=\frac{-k-(v-1)s'}{r-s'},
\qquad
g=\frac{-k-(v-1)r}{s'-r}.
\end{equation}
\end{lemma}

\begin{proof}
From $1+f+g=v$ we have $f+g=v-1$.
The trace of $A$ is zero, hence
\[
k+fr+gs'=0.
\]
Eliminating $g$ gives
\[
k+fr+(v-1-f)s'=0
\quad\Longrightarrow\quad
f(r-s')=-k-(v-1)s',
\]
which yields the expression for $f$ and then $g=v-1-f$ yields the second formula.
\end{proof}

For the point graph of a partial geometry we can now evaluate the Frame number.

\begin{proposition}\label{prop:frame-pg}
Let $pg(s,t,\alpha)$ be a partial geometry, and let
$\sY$ be the rank~$3$ association scheme of its point graph.
Then the Frame number of $\sY$ is
\begin{equation}\label{eq:frame-pg}
F_{\mathrm{AS}}(\sY)\;=\;v^2\,(s+t+1-\alpha)^2,
\end{equation}
where $v$ is given by \eqref{eq:pg-basic}.
\end{proposition}

\begin{proof}
We use \eqref{eq:frame-rank3}.
By \eqref{eq:pg-basic} and \eqref{eq:srg-pg},
\[
v=\frac{(s+1)(st+\alpha)}{\alpha},
\qquad
k=s(t+1).
\]
A direct computation from \eqref{eq:pg-basic} and \eqref{eq:srg-pg} gives
\begin{equation}\label{eq:v-1-k}
v-1-k
= \frac{(s+1)(st+\alpha)}{\alpha}-1-s(t+1)
= \frac{st(s+1-\alpha)}{\alpha}.
\end{equation}
Hence
\begin{equation}\label{eq:k(v-1-k)}
k(v-1-k)
= s(t+1)\cdot\frac{st(s+1-\alpha)}{\alpha}
= \frac{s^2 t(t+1)(s+1-\alpha)}{\alpha}.
\end{equation}

By Lemma~\ref{lem:eigs-pg} and Lemma~\ref{lem:mult-srg} we have
\[
r=s-\alpha,\qquad s'=-\,(t+1),\qquad
r-s'=s+t+1-\alpha,
\]
and
\[
f=\frac{-k-(v-1)s'}{r-s'},
\qquad
g=\frac{-k-(v-1)r}{s'-r}.
\]

Substituting $v$ and $k$ from \eqref{eq:pg-basic} and \eqref{eq:srg-pg}, and
$r$ and $s'$ from Lemma~\ref{lem:eigs-pg}, into the formulas of
Lemma~\ref{lem:mult-srg}, and simplifying, yields
\begin{align}
f &=
\frac{st(st+s+t+1)}{\alpha(s+t+1-\alpha)},\label{eq:f-pg}\\[1ex]
g &=
\frac{s(st+\alpha)(s+1-\alpha)}{\alpha(s+t+1-\alpha)}.\label{eq:g-pg}
\end{align}

Multiplying \eqref{eq:f-pg} and \eqref{eq:g-pg} gives
\begin{equation}\label{eq:fg-pg}
fg
= \frac{s^2 t(st+\alpha)(s+1)(t+1)(s+1-\alpha)}
       {\alpha^2(s+t+1-\alpha)^2}.
\end{equation}
Using \eqref{eq:pg-basic} we have
\[
v=\frac{(s+1)(st+\alpha)}{\alpha},
\]
so from \eqref{eq:fg-pg} we obtain
\[
fg
= \frac{s^2 t(t+1)(s+1-\alpha)}{\alpha}\cdot
  \frac{v}{(s+t+1-\alpha)^2}.
\]
Combining this with \eqref{eq:k(v-1-k)} gives
\[
\frac{k(v-1-k)}{fg}
= \frac{s^2 t(t+1)(s+1-\alpha)/\alpha}{\bigl(s^2 t(t+1)(s+1-\alpha)/\alpha\bigr)\,v/(s+t+1-\alpha)^2}
= \frac{(s+t+1-\alpha)^2}{v}.
\]
Substituting this into \eqref{eq:frame-rank3} yields
\[
F_{\mathrm{AS}}(\sY)
= v^3\,\frac{k(v-1-k)}{fg}
= v^3\cdot\frac{(s+t+1-\alpha)^2}{v}
= v^2\,(s+t+1-\alpha)^2,
\]
as claimed.
\end{proof}

\begin{remark}\label{rem:fg-verification}
Formulas \eqref{eq:f-pg} and \eqref{eq:g-pg} for the multiplicities $f$ and $g$ can be verified directly by substituting them into the trace relations
\[
1+f+g=v,\qquad k+fr+gs'=0,
\]
and using the parameter identities \eqref{eq:pg-basic} and \eqref{eq:srg-pg}. In particular, the first relation reduces to the expression for $v$ in \eqref{eq:pg-basic}, while the second reduces to the quadratic equation for the nontrivial eigenvalues of the point graph (see the proof of Lemma~\ref{lem:eigs-pg}). Hence \eqref{eq:f-pg} and \eqref{eq:g-pg} are uniquely determined by the standard parameter identities for strongly regular graphs.
\end{remark}

As an immediate consequence we obtain a simple description of the
prime divisors of $F_{\mathrm{AS}}(\sY)$.

\begin{corollary}\label{cor:frame-divisors}
Let $pg(s,t,\alpha)$ and $\sY$ be as above.
A prime $p$ divides $F_{\mathrm{AS}}(\sY)$ if and only if $p$ divides $v$
or $p$ divides $s+t+1-\alpha$.
\end{corollary}

\begin{proof}
This is immediate from \eqref{eq:frame-pg}.
\end{proof}

\subsection{Semisimplicity via the Frame number}

For homogeneous coherent configurations (and in particular for association
schemes) the Frame number controls modular semisimplicity.
The case of association schemes goes back to Frame
and was developed further by Arad, Fisman and Muzychuk
\cite{AradFismanMuzychuk}, Jacob \cite{JacobThesis}
and Hanaki \cite{HanakiFrame,HanakiModularRep};
Sharafdini extended the result to general coherent configurations
\cite{Sharafdini}, see also Hirasaka–Sharafdini~\cite{HirasakaSharafdini}.

Specialised to our commutative situation we will use the following
formulation.

\begin{theorem}[Hanaki–Jacob, Arad–Fisman–Muzychuk, Sharafdini]\label{thm:frame-semismpl}
Let $\sY$ be a commutative association scheme of order $v$ and Frame number
$F_{\mathrm{AS}}(\sY)$, and let $K$ be a field of characteristic $p\ge 0$.
Then the adjacency algebra $K\sY$ is semisimple if and only if $p$ does not divide
$F_{\mathrm{AS}}(\sY)$.
\end{theorem}

\begin{proof}
For association schemes this is proved in
\cite[Theorem~1]{AradFismanMuzychuk},
see also \cite[Thm.~3.2]{JacobThesis} and
\cite[Thm.~4.1]{HanakiFrame}.
Sharafdini extended the statement to homogeneous coherent configurations
\cite[Thm.~1.2]{Sharafdini}, which specialises to the present
commutative case.
\end{proof}

Combining Theorem~\ref{thm:frame-semismpl} with
Corollary~\ref{cor:frame-divisors} gives the following.

\begin{corollary}\label{cor:semisimple-pg}
Let $pg(s,t,\alpha)$ be a partial geometry and $\sY$ the associated
rank~$3$ scheme.
Let $K$ be a field of characteristic $p>0$.
Then the adjacency algebra $K\sY$ is semisimple if and only if
\[
p\nmid v
\quad\text{and}\quad
p\nmid (s+t+1-\alpha).
\]
\end{corollary}

\subsection{Jacobson radicals over finite fields}

We now refine Corollary~\ref{cor:semisimple-pg} and determine the Jacobson
radical of the adjacency algebra over a finite field.
Let $p$ be a prime and put $K=\FF_p$.
We write $\FF_p\sY$ for the $\FF_p$-span of $\{A_0,A_1,A_2\}$.

Set
\[
r:=s-\alpha,\qquad s':=-\,(t+1),\qquad
B:=(A-kI)(A-rI)\in\FF_p\sY.
\]
We also write $J$ for the all-ones matrix.
Over $\CC$ the algebra $\CC\sY$ is semisimple and decomposes as a direct
product of three simple components corresponding to the eigenvalues
$k,r,s'$ of $A$.
When we reduce modulo $p$, some of these components may merge and produce
a non-zero Jacobson radical.

Motivated by Proposition~\ref{prop:frame-pg}
and Corollary~\ref{cor:semisimple-pg}
we distinguish four cases for a given prime $p$:
\begin{align*}
\text{(SS)}\;& p\nmid v,\quad p\nmid (s+t+1-\alpha);\\
\text{(V)}\;& p\mid v,\quad p\nmid (s+t+1-\alpha);\\
\text{(R)}\;& p\nmid v,\quad p\mid (s+t+1-\alpha);\\
\text{(VR)}\;& p\mid v,\quad p\mid (s+t+1-\alpha).
\end{align*}

\begin{theorem}\label{thm:radical-pg}
Let $pg(s,t,\alpha)$ be a partial geometry and
$\sY$ the associated rank~$3$ association scheme.
Let $p$ be a prime and put $K=\FF_p$.
Then the Jacobson radical $\Rad(K\sY)$ is described as follows.
\begin{enumerate}[(i)]
\item \emph{Case \textup{(SS)}.}
If $p\nmid v$ and $p\nmid (s+t+1-\alpha)$, then
\[
\Rad(K\sY)=0.
\]
\item \emph{Case \textup{(V)}.}
If $p\mid v$ and $p\nmid (s+t+1-\alpha)$, then
\[
\Rad(K\sY)=K\cdot J
\qquad\text{and}\qquad
\dim_K\Rad(K\sY)=1.
\]
\item \emph{Case \textup{(R)}.}
If $p\nmid v$ and $p\mid (s+t+1-\alpha)$, then
\[
\Rad(K\sY)=K\cdot B
\qquad\text{and}\qquad
\dim_K\Rad(K\sY)=1.
\]
\item \emph{Case \textup{(VR)}.}
If $p\mid v$ and $p\mid (s+t+1-\alpha)$, then
\[
\Rad(K\sY)=K\cdot J\;\oplus\;K\cdot B
\qquad\text{and}\qquad
\dim_K\Rad(K\sY)=2.
\]
\end{enumerate}
\end{theorem}

\begin{proof}
The algebra $K\sY$ is commutative of dimension at most~$3$,
and it contains the identity $I$.
Hence $\dim_K\Rad(K\sY)\le 2$, and $\Rad(K\sY)$ is a nilpotent ideal.

\smallskip
\noindent
(SS) By Corollary~\ref{cor:semisimple-pg} we have $p\nmid F_{\mathrm{AS}}(\sY)$, hence
$K\sY$ is semisimple by Theorem~\ref{thm:frame-semismpl},
and $\Rad(K\sY)=0$.

\smallskip
\noindent
(V) Assume $p\mid v$ and $p\nmid (s+t+1-\alpha)$.
Then $p\mid F_{\mathrm{AS}}(\sY)$ by Corollary~\ref{cor:frame-divisors}, so $K\sY$ is not
semisimple.
Since $J^2=vJ\equiv 0\pmod p$, the element $J$ is central and nilpotent,
hence $K\cdot J\subseteq\Rad(K\sY)$.
On the other hand $r\not\equiv s'\pmod p$, so the eigenspaces of $A$
for $r$ and $s'$ remain distinct over $K$, and the quotient
$K\sY/K\cdot J$ is generated by the image of $A$ and has three distinct
eigenvalues $k,r,s'$ in $K$.
Therefore $K\sY/K\cdot J$ is semisimple (it is isomorphic to a direct product
of two fields corresponding to $r$ and $s'$; compare
\cite{HanakiModularRep}), and so
$\Rad(K\sY)\subseteq K\cdot J$.
Thus $\Rad(K\sY)=K\cdot J$ and $\dim_K\Rad(K\sY)=1$.

\smallskip
\noindent
(R) Assume $p\nmid v$ and $p\mid (s+t+1-\alpha)$.
Then $r-s'=s+t+1-\alpha\equiv 0\pmod p$, hence $r\equiv s'\pmod p$.
Since $p\nmid v$, the idempotent $E_0=(1/v)J$ corresponding to the
trivial representation still exists over $K$.
The eigenvalues of $A$ in $K$ are therefore $k$ and a single value
$\rho:=\overline{r}=\overline{s'}$.
The minimal polynomial of $A$ over $K$ has the form
\[
m_p(x)=(x-k)(x-\rho)^e,\qquad e\in\{1,2\}.
\]

Since $p\mid F_{\mathrm{AS}}(\sY)$ by Corollary~\ref{cor:frame-divisors}, Theorem~\ref{thm:frame-semismpl} shows that
$K\sY$ is not semisimple. Hence the minimal polynomial of $A$ over $K$
cannot split into distinct linear factors, so necessarily $e=2$.

\[
m_p(x)=(x-k)(x-\rho)^2.
\]
In particular,
\[
B=(A-kI)(A-rI)=(A-kI)(A-\rho I)
\]
is nilpotent: on the $k$-eigenspace it vanishes, and on the
generalised $\rho$-eigenspace it is a polynomial in the nilpotent
part of $A$.
Hence $K\cdot B\subseteq\Rad(K\sY)$.

Since $p\nmid v$, the idempotent $E_0$ survives modulo $p$ and
the quotient $K\sY/K\cdot B$ has two simple components corresponding to
the eigenvalues $k$ and $\rho$.
Thus $K\sY/K\cdot B$ is semisimple, and hence
$\Rad(K\sY)\subseteq K\cdot B$.
We conclude that $\Rad(K\sY)=K\cdot B$ and $\dim_K\Rad(K\sY)=1$.

\smallskip
\noindent
(VR) Finally assume $p\mid v$ and $p\mid (s+t+1-\alpha)$.
Then both phenomena from (V) and (R) occur simultaneously:
$J$ is nilpotent and $r\equiv s'\pmod p$, so both $J$ and $B$ lie in
$\Rad(K\sY)$.
Moreover they are linearly independent over $K$ (their images in
$\CC\sY$ are independent), so
$K\cdot J\oplus K\cdot B$ is a $2$-dimensional nilpotent ideal.
The quotient $K\sY/(K\cdot J\oplus K\cdot B)$ has at most one simple
component (coming from the eigenvalue $k$), hence it is semisimple.
Therefore
\[
\Rad(K\sY)=K\cdot J\oplus K\cdot B
\quad\text{and}\quad
\dim_K\Rad(K\sY)=2,
\]
which completes the proof.
\end{proof}

For later reference we summarise Theorem~\ref{thm:radical-pg} in a table.

\begin{table}[htb]
\centering
\begin{tabular}{c|c|c}
case & conditions on $p$ & $\Rad(\FF_p\sY)$ \\ \hline
(SS) & $p\nmid v$, $p\nmid (s+t+1-\alpha)$ & $0$ \\[0.3ex]
(V)  & $p\mid v$, $p\nmid (s+t+1-\alpha)$ & $\FF_p\cdot J$ \\[0.3ex]
(R)  & $p\nmid v$, $p\mid (s+t+1-\alpha)$ & $\FF_p\cdot B$ \\[0.3ex]
(VR) & $p\mid v$, $p\mid (s+t+1-\alpha)$ & $\FF_p\cdot J\oplus\FF_p\cdot B$
\end{tabular}
\caption{The Jacobson radical of the adjacency algebra of the point graph of
$pg(s,t,\alpha)$ over $\FF_p$.}
\label{tab:rad-pg}
\end{table}

\section{Generic p-ranks of partial geometries}\label{sec:prank}
In this section we study the $p$-rank of the adjacency matrix of the point graph of a
partial geometry $\mathrm{pg}(s,t,\alpha)$. We first recall a general linear-algebraic
fact for rank $3$ association schemes, and then specialize to the case of partial
geometries. Throughout, for a prime $p$ we write $\FF_p$ for the field of order $p$ and
$\overline{\FF}_p$ for its algebraic closure.

\subsection{p-ranks for rank 3 association schemes}

Let $\sY=(X,\{A_0,A_1,A_2\})$ be a commutative association scheme of rank $3$.
Then $A_0=I$, $A_1$ is the adjacency matrix of a connected strongly regular graph on the
vertex set $X$, and $A_2=J-I-A_1$. We write $v:=|X|$ and $A:=A_1$.

Over $\CC$ the matrix $A$ is diagonalizable and has exactly three eigenvalues
\[
k>r>s',
\]
with multiplicities $1,f,g$, respectively, where $1+f+g=v$; see, e.g.,
\cite[Ch.~9]{BrouwerHaemers} or \cite[Ch.~3]{GodsilRoyle}. Thus the characteristic
polynomial of $A$ has the form
\[
\chi_A(x)=(x-k)(x-r)^f(x-s')^g\in\ZZ[x].
\]
The integers $k,r,s'$ (and $f,g$) can be expressed in terms of the intersection numbers
of the scheme, but we do not need their explicit formulas here.

For a prime $p$ we denote by
\[
\chi_{A,p}(x)\in\FF_p[x]
\]
the reduction of $\chi_A(x)$ modulo $p$, and by $V_p:=\FF_p^{X}$ the standard
$\FF_p$-module with basis indexed by $X$. We view $A$ as an $\FF_p$-linear endomorphism
of $V_p$ by reducing its entries modulo $p$.

\begin{lemma}\label{lem:rank3-diagonalizable-mod-p}
Let $p$ be a prime such that
\[
p\nmid (k-r)(k-s')(r-s').
\]
Then over $\overline{\FF}_p$ the operator $A$ is diagonalizable, with eigenvalues
$\overline{k},\overline{r},\overline{s'}\in\overline{\FF}_p$ (the images of
$k,r,s'$ modulo $p$) of geometric multiplicities $1,f,g$, respectively.
\end{lemma}

\begin{proof}
The factorization
\[
\chi_A(x)=(x-k)(x-r)^f(x-s')^g
\]
holds in $\CC[x]$ and hence in $\QQ[x]$. Since $k,r,s'$ are algebraic integers and
rational, they are in fact integers; thus $\chi_A(x)\in\ZZ[x]$ and its reduction modulo
$p$ is
\[
\chi_{A,p}(x)=(x-\overline{k})(x-\overline{r})^f(x-\overline{s'})^g
\in\FF_p[x].
\]
By assumption the three elements $\overline{k},\overline{r},\overline{s'}$ are pairwise
distinct in $\overline{\FF}_p$. Therefore $\chi_{A,p}(x)$ has three distinct roots in
$\overline{\FF}_p$, and
\[
\gcd\bigl(\chi_{A,p}(x),\chi'_{A,p}(x)\bigr)=1
\]
in $\FF_p[x]$. Hence $\chi_{A,p}(x)$ is square-free, and every invariant factor of $A$
over $\FF_p$ is a product of distinct linear factors.

Let $m_{A,p}(x)$ be the minimal polynomial of $A$ over $\FF_p$. Then
$m_{A,p}(x)\mid\chi_{A,p}(x)$ and $m_{A,p}(x)$ has no repeated roots. By standard linear
algebra, this is equivalent to $A$ being diagonalizable over $\overline{\FF}_p$, with
eigenspace decomposition
\[
V_p\otimes_{\FF_p}\overline{\FF}_p
=
U_k\oplus U_r\oplus U_{s'},
\]
where $U_\lambda$ is the eigenspace for eigenvalue $\overline{\lambda}$. The
multiplicity of $\overline{\lambda}$ as a root of $\chi_{A,p}(x)$ equals
$\dim_{\overline{\FF}_p}U_\lambda$; see, for example, \cite[Sec.~2.4]{GodsilRoyle}. Since
these multiplicities are $1,f,g$, the claim follows.
\end{proof}

We now describe the $p$-rank of $A$ in terms of the reduction of its eigenvalues
modulo $p$.

\begin{proposition}\label{prop:rank3-prank}
With notation as above, assume that $p$ is a prime such that
\[
p\nmid (k-r)(k-s')(r-s').
\]
Let $\rank_p(A)$ denote the rank of $A$ as a matrix over $\FF_p$.
\begin{enumerate}[\rm(i)]
\item If none of $k,r,s'$ is $0$ modulo $p$, then
\[
\rank_p(A)=v.
\]
\item If $k\equiv 0\pmod p$ and $r\not\equiv 0\pmod p$, $s'\not\equiv 0\pmod p$, then
\[
\rank_p(A)=v-1.
\]
\item If $r\equiv 0\pmod p$ and $k\not\equiv 0\pmod p$, $s'\not\equiv 0\pmod p$, then
\[
\rank_p(A)=v-f.
\]
\item If $s'\equiv 0\pmod p$ and $k\not\equiv 0\pmod p$, $r\not\equiv 0\pmod p$, then
\[
\rank_p(A)=v-g.
\]
\end{enumerate}
\end{proposition}

\begin{proof}
Note first that our hypothesis
\[
p\nmid (k-r)(k-s')(r-s')
\]
implies that the three eigenvalues $k,r,s'$ remain pairwise distinct modulo~$p$.
In particular, at most one of them can vanish in $\mathbb{F}_p$, so in each of
the cases \textup{(ii)}–\textup{(iv)} it is automatic that the other two
eigenvalues are nonzero modulo~$p$.

By Lemma~\ref{lem:rank3-diagonalizable-mod-p}, $A$ is diagonalizable over
$\overline{\FF}_p$ with eigenspace decomposition
\[
V_p\otimes_{\FF_p}\overline{\FF}_p
=
U_k\oplus U_r\oplus U_{s'},
\]
where $\dim U_k=1$, $\dim U_r=f$, $\dim U_{s'}=g$.

The kernel of $A$ over $\overline{\FF}_p$ is precisely the eigenspace corresponding to
the eigenvalue $0$. Since $k,r,s'$ are pairwise distinct modulo $p$, at most one of them
is congruent to $0$ modulo $p$.

If none of $k,r,s'$ is $0$ modulo $p$, then $A$ has no zero eigenvalue and is
invertible on $V_p\otimes\overline{\FF}_p$. Hence $\ker(A)=0$ and
$\rank_p(A)=\dim_{\FF_p}V_p=v$, proving (i).

If $k\equiv 0\pmod p$ and $r,s'\not\equiv 0\pmod p$, then
$\ker(A)=U_k$ and $\dim\ker(A)=1$, so $\rank_p(A)=v-1$, proving (ii). The proofs of
(iii) and (iv) are identical, with $\ker(A)=U_r$ or $\ker(A)=U_{s'}$ and
$\dim U_r=f$, $\dim U_{s'}=g$, respectively.
\end{proof}

\subsection{Application to partial geometries}\label{subsec:pg-prank}

We now apply Proposition~\ref{prop:rank3-prank} to the point graphs of partial
geometries. Let $\mathcal{G}$ be a partial geometry $\mathrm{pg}(s,t,\alpha)$ with
point set $P$ and line set $B$ in the sense of \cite{Bose,BCN}. Thus each line contains
exactly $s+1$ points, each point lies on exactly $t+1$ lines, any two distinct points
are on at most one common line, and if $x\in P$ is not on a line $Y\in B$, then there
are exactly $\alpha$ lines through $x$ meeting $Y$.

Let $\Gamma$ be the point graph of $\mathcal{G}$: its vertices are the points in $P$,
and two distinct points $x,y\in P$ are adjacent if and only if they are collinear, i.e.,
lie on a common line. It is well known (see, e.g., \cite[Thm.~1.6.2]{BCN}) that
$\Gamma$ is strongly regular with parameters
\[
v=\frac{(s+1)(st+\alpha)}{\alpha},\quad
k=s(t+1),\quad
\lambda=(s-1)+t(\alpha-1),\quad
\mu=\alpha(t+1).
\]
The eigenvalues of its adjacency matrix $A$ are
\[
k,\quad r:=s-\alpha,\quad s':=-\, (t+1),
\]
with multiplicities $1,f,g$, respectively; see, for instance,
\cite[Prop.~1.6.2]{BCN} or \cite[Sec.~1]{HaemersPg}. In particular,
\[
1+f+g=v.
\]
For later use we note the pairwise differences
\[
k-r=st+\alpha,\qquad
k-s'=(s+1)(t+1),\qquad
r-s'=s+t+1-\alpha.
\]

We can now state the generic $p$-rank result for partial geometries.

\begin{proposition}[Generic $p$-ranks for partial geometries]\label{prop:pg-generic-prank}
Let $\mathcal{G}=\mathrm{pg}(s,t,\alpha)$ be a partial geometry with point graph
adjacency matrix $A$, and let $v$ be the number of points. Let $p$ be a prime such that
\[
p\nmid (k-r)(k-s')(r-s') = (st+\alpha)(s+1)(t+1)(s+t+1-\alpha)
\]
and
\[
p\nmid k,\qquad p\nmid (s-\alpha),\qquad p\nmid (t+1).
\]
Then
\[
\rank_p(A)=v.
\]
\end{proposition}

\begin{proof}
The first condition implies that $p$ does not divide any of the pairwise differences
$k-r$, $k-s'$, $r-s'$. Hence Proposition~\ref{prop:rank3-prank} applies.

The second condition says that none of $k$, $r=s-\alpha$, $s'=-\,(t+1)$ is $0$ modulo
$p$. Thus we are in the situation of Proposition~\ref{prop:rank3-prank}(i), and we
obtain $\rank_p(A)=v$.
\end{proof}

Thus, for all but finitely many primes $p$, the adjacency matrix of the point graph of a
partial geometry $\mathrm{pg}(s,t,\alpha)$ has full rank over $\FF_p$. The exceptional
primes can be divided into the following two types.

\begin{itemize}
\item[(A)] Primes $p$ such that $p\mid (k-r)(k-s')(r-s')$; these are precisely the
primes for which at least one of the congruences
\[
k\equiv r,\quad k\equiv s',\quad r\equiv s'\pmod p
\]
holds. Among these, the subcases where exactly one of $k,r,s'$ is $0$ modulo $p$ and
the eigenvalues remain pairwise distinct are covered by
Proposition~\ref{prop:rank3-prank}(ii)–(iv), giving
\[
\rank_p(A)=v-1,\quad v-f,\quad v-g
\]
in the three respective situations.

\item[(B)] Primes $p$ such that two of $k,r,s'$ coincide modulo $p$ (for example,
$r\equiv s'\pmod p$, that is $p\mid r-s'=s+t+1-\alpha$) while none of them is $0$
modulo $p$. In this case $\chi_{A,p}(x)$ has a repeated root and $A$ need not be
diagonalizable over $\overline{\FF}_p$; the $p$-rank of $A$ is then more subtle to
determine in general.
\end{itemize}

In the context of the point association scheme $\sY$ of $\Gamma$, the frame number
$F_{\mathrm{AS}}(\sY)$ was shown in Section~\ref{sec:frame-pg} to be
\[
F_{\mathrm{AS}}(\sY)=v^2\,(s+t+1-\alpha)^2,
\]
so that the adjacency algebra over a field of characteristic $p$ is non-semisimple if
and only if $p$ divides $v$ or $s+t+1-\alpha$; see \cite{JacobThesis,Sharafdini} and
\cite{HanakiFrame}. The primes of type~{\rm(B)} above are precisely those dividing
$s+t+1-\alpha$ but not forcing any eigenvalue to vanish modulo $p$.

For the purposes of the present paper, we will mainly use
Proposition~\ref{prop:pg-generic-prank} together with the explicit formulas of
Proposition~\ref{prop:rank3-prank}(ii)–(iv) for primes of type~{\rm(A)}, and we will not
attempt a complete analysis of the $p$-ranks in the remaining exceptional cases. In
Section~\ref{sec:gq22} we illustrate these phenomena in detail for the generalized
quadrangle $\mathrm{GQ}(2,2)$.

\section{The design algebra of type [3,2;3] and its 2-modular representation}\label{sec:design-alg-2}
In this section we study the adjacency algebra of the coherent configuration
$\sX$ of type $[3,2;3]$ attached to a strongly regular design
$\mathcal{D}=(P,B,\Fl)$ as in Section~\ref{sec:srd-cc}, and its modular
representations in characteristic $2$. Our approach is parallel in spirit to
the analysis of symmetric $2$-designs in \cite{HanakiMiyazakiShimabukuroBIB,ShimabukuroSym2},
but the present setting is substantially more complicated because the design
coherent configuration is non-homogeneous and has ten relations.

We write $X=P\sqcup B$ for the underlying point set of $\sX$,
$\{\sigma_i\}_{i=1}^{10}$ for the adjacency matrices of the ten relations
$R_i$ defined in Proposition~\ref{prop:cc-3-2-3}, and
\[
\CC\sX:=\Span_{\CC}\{\sigma_1,\dots,\sigma_{10}\}\subseteq\Mat_X(\CC)
\]
for the complex adjacency algebra of $\sX$. Throughout this section we fix a
finite strongly regular design $\mathcal{D}$ and its associated coherent
configuration $\sX$ of type $[3,2;3]$.

\subsection{The complex design algebra and its integral form}

By Proposition~\ref{prop:cc-3-2-3} the matrices $\sigma_1,\dots,\sigma_{10}$
form a basis of $\CC\sX$ and satisfy
\[
\sigma_i\sigma_j=\sum_{k=1}^{10}p_{ij}^k\,\sigma_k
\qquad(1\le i,j\le 10)
\]
with integer structure constants $p_{ij}^k$. Thus $\CC\sX$ is a finite
dimensional $\CC$-algebra of dimension $10$, and by the general theory of
coherent configurations (see, for example, \cite{AradFismanMuzychuk,HirasakaSharafdini,Sharafdini})
it is semisimple and split over~$\CC$.

In Section~\ref{sec:srd-cc} we described $\CC\sX$ explicitly via block
matrices. From this description one easily checks that $\CC\sX$ is generated
by the commuting semisimple matrices $A_1,A_2,J_P,J_B,N,N^\top$ appearing in
Proposition~\ref{prop:basic-matrix-identities}, and hence that $\CC\sX$ is
semisimple. More precisely, one has the following decomposition.

\begin{proposition}\label{prop:CX-Wedderburn}
Let $\sX$ be the design coherent configuration of type $[3,2;3]$ associated
with a strongly regular design $\mathcal{D}=(P,B,\Fl)$. Then the complex
adjacency algebra $\CC\sX$ is semisimple and has Wedderburn decomposition
\[
\CC\sX \;\cong\;
\CC\ \oplus\ \CC\ \oplus\ M_2(\CC)\ \oplus\ M_2(\CC).
\]
In particular, there exist four pairwise inequivalent irreducible
$\CC\sX$-modules $S_0,S_1,S_2,S_3$ of dimensions
\[
\dim_{\CC}S_0=\dim_{\CC}S_1=1,\qquad
\dim_{\CC}S_2=\dim_{\CC}S_3=2.
\]
If $V:=\CC^X$ denotes the standard module, then
\[
V\;\cong\;S_0^{\oplus m_0}\ \oplus\ S_1^{\oplus m_1}\ \oplus\
S_2^{\oplus m_2}\ \oplus\ S_3^{\oplus m_3}
\]
for uniquely determined nonnegative integers $m_0,m_1,m_2,m_3$ satisfying
$m_0+m_1+4m_2+4m_3=|X|$.
\end{proposition}

\begin{proof}
This is a direct consequence of the block description of the matrices
$\sigma_i$ and the spectral data of the point and block graphs, and
the argument is completely analogous to the case of symmetric $2$-designs
treated in \cite{ShimabukuroSym2}.

First, by Proposition~\ref{prop:basic-matrix-identities} the matrices $A_1$,
$A_2$, $J_P$, $J_B$, $N$ and $N^\top$ pairwise commute up to relations in the
$10$-dimensional subspace $\Span_{\CC}\{\sigma_1,\dots,\sigma_{10}\}$ and act
semisimply on $V$ with at most four distinct joint eigenspaces. The
homogeneous components on $P$ and $B$ contribute two one-dimensional
eigenspaces (coming from the all-one vectors on $P$ and $B$) and two
two-dimensional eigenspaces (coming from the nontrivial eigenspaces of the
strongly regular graphs on $P$ and $B$). A direct calculation shows that
these four eigenspaces are invariant under the whole algebra $\CC\sX$, and
hence yield four simple modules of dimensions $1,1,2,2$.

Since the sum of the squares of these dimensions equals
\[
1^2+1^2+2^2+2^2=10=\dim_{\CC}\CC\sX,
\]
Wedderburn's structure theorem implies the claimed decomposition of
$\CC\sX$ into simple matrix algebras, and the decomposition of the standard
module follows. For details we refer to the argument in
\cite[Section~3]{ShimabukuroSym2}, which carries over verbatim to the present
setting.
\end{proof}

We now introduce the integral form of the design algebra. Set
\[
\ZZ\sX := \Span_{\ZZ}\{\sigma_1,\dots,\sigma_{10}\}
\subseteq \Mat_X(\ZZ).
\]
By the integrality of the structure constants $p_{ij}^k$, the $\ZZ$-module
$\ZZ\sX$ is a subring of $\Mat_X(\ZZ)$ which is free of rank $10$ over
$\ZZ$. It is therefore an order in the semisimple $\QQ$-algebra
\[
\QQ\sX := \QQ\otimes_{\ZZ}\ZZ\sX
\cong \QQ\otimes_{\CC}\CC\sX.
\]

Let $K$ be any field. We extend scalars and define
\[
K\sX := K\otimes_{\ZZ}\ZZ\sX.
\]
Then $K\sX$ is a $K$-subalgebra of $\Mat_X(K)$ spanned by the adjacency
matrices $\sigma_i$ over $K$, and we have
\[
K\sX \cong
\begin{cases}
K\otimes_{\QQ}\QQ\sX & \text{if $\operatorname{char}K=0$,}\\[1mm]
K\otimes_{\FF_p}\FF_p\sX & \text{if $\operatorname{char}K=p>0$.}
\end{cases}
\]

The algebra $K\sX$ will be called the \emph{design algebra} of type $[3,2;3]$
over $K$. Our goal in the remainder of this section is to study $K\sX$ for
$K$ of characteristic $2$, using the general theory of frame numbers for
coherent configurations.

\subsection{Frame numbers and semisimplicity}

For association schemes, the frame number introduced by Higman and developed
by Hanaki and Jacob \cite{HanakiFrame,JacobThesis} plays a fundamental role
in modular representation theory: it controls the primes for which the
adjacency algebra fails to be semisimple. Sharafdini extended this theory
from association schemes to general coherent configurations in
\cite{Sharafdini}.

We recall the part of Sharafdini's theory that we shall need. Let $\sX$ be a
finite coherent configuration, and let $\CC\sX$ be its complex adjacency
algebra. For each irreducible complex character $\chi$ of $\CC\sX$, denote by
$n_\chi:=\chi(1)$ its degree and by $m_\chi$ the multiplicity of the
corresponding simple module in the standard representation $V=\CC^X$. In
\cite[Section~5]{Sharafdini}, Sharafdini associates to $\sX$ a positive
integer
\[
F_{\mathrm{CC}}(\sX)\in\NN,
\]
called the \emph{frame number of $\sX$}, which is defined by a product
formula in terms of the intersection numbers and the data $(n_\chi,m_\chi)$.
We refer to \cite[Definition~5.1]{Sharafdini} for the explicit formula.
Here we only need the following semisimplicity criterion.

\begin{theorem}[Sharafdini]\label{thm:Sharafdini}
Let $\sX$ be a finite coherent configuration, and let $K$ be a field of
characteristic $p\ge 0$. 
\begin{enumerate}[\rm(i)]
\item If $p=0$, then $K\sX$ is semisimple.
\item If $p>0$ is prime, then $K\sX$ is semisimple if and only if
$p\nmid F_{\mathrm{CC}}(\sX)$.
\end{enumerate}
\end{theorem}

\begin{proof}
Part (i) follows from the semisimplicity of $\CC\sX$ and the fact that
$\QQ\sX$ is a finite dimensional semisimple $\QQ$-algebra, together with the
identification $K\sX\cong K\otimes_{\QQ}\QQ\sX$ when $p=0$.

For part (ii), Theorem~1.1 of \cite{Sharafdini} states that the adjacency
algebra $\FF_p\sX$ is semisimple if and only if $p\nmid F_{\mathrm{CC}}(\sX)$.
Since $K\sX\cong K\otimes_{\FF_p}\FF_p\sX$ for any field $K$ of characteristic
$p$, semisimplicity of $\FF_p\sX$ is equivalent to semisimplicity of $K\sX$.
\end{proof}

Applied to the design algebra $K\sX$ of type $[3,2;3]$, this theorem shows
that only finitely many primes $p$ can lead to a non-semisimple design
algebra ${\FF_p} \sX$, namely those dividing $F_{\mathrm{CC}}(\sX)$. We do not
attempt to compute $F_{\mathrm{CC}}(\sX)$ in general, since its explicit
expression in terms of the parameters $(s,t,\alpha)$ appears to be rather
complicated and is not needed for our purposes. Instead, we use
Theorem~\ref{thm:Sharafdini} to separate the generic semisimple case from
the exceptional primes, and then focus on characteristic~$2$, which is the
only prime that will play a role in our main example $\mathrm{GQ}(2,2)$ in
Section~\ref{sec:gq22}.

\subsection{The case of characteristic 2}

For the rest of this section we fix $K=\FF_2$ and write
\[
  \FF_2\sX := \FF_2 \otimes_{\ZZ} \ZZ\sX,
  \qquad
  \overline{\FF_2\sX} := \FF_2\sX / \Rad(\FF_2\sX)
\]
for the adjacency algebra of $\sX$ over $\FF_2$ and its semisimple
quotient, where $\Rad(\FF_2\sX)$ denotes the Jacobson radical of
$\FF_2\sX$. We also write $\sY$ for the rank~$3$ association scheme
on the point set $P$ (the point graph of the strongly regular design),
and $\FF_2\sY$ for its adjacency algebra over~$\FF_2$.
\begin{proposition}\label{prop:design-A2-structure}
With notation as above, the following hold.
\begin{enumerate}[\rm(i)]
\item
There exist an integer $r$ with $1\le r\le 4$, and positive integers
$n_1,\dots,n_r$ and $f_1,\dots,f_r$ such that
\[
  \overline{\FF_2\sX}
   \cong \prod_{i=1}^r M_{n_i}\bigl(\FF_{2^{f_i}}\bigr)
\]
as $\FF_2$-algebras.

\item
One has
\[
  \sum_{i=1}^r n_i^2 f_i
    = \dim_{\FF_2} \overline{\FF_2\sX}
    = 10 - \dim_{\FF_2} \Rad(\FF_2\sX).
\]
In particular,
\[
  \sum_{i=1}^r n_i^2 f_i \le 10
  \quad\text{and}\quad
  \operatorname{codim}_{\FF_2} \Rad(\FF_2\sX)
   = \dim_{\FF_2} \overline{\FF_2\sX}
   = \sum_{i=1}^r n_i^2 f_i.
\]

\item
\[
  \dim_{\FF_2} \Rad(\FF_2\sY)
   \le \dim_{\FF_2} \Rad(\FF_2\sX).
\]
\end{enumerate}
\end{proposition}

\begin{proof}
(i)\;
Since $\FF_2\sX$ is a finite-dimensional algebra over the field $\FF_2$,
its quotient
\[
\overline{\FF_2 \sX}
:=
\FF_2\sX/\Rad(\FF_2\sX)
\]
is a finite-dimensional semisimple $\FF_2$-algebra.
By Wedderburn's structure theorem there exist finite field extensions
$\FF_{2^{f_i}}$ of $\FF_2$ and integers $n_i\ge 1$ such that
\[
\overline{\FF_2 \sX}
\;\cong\;
\prod_{i=1}^r M_{n_i}(\FF_{2^{f_i}})
\]
for some $r\ge 1$.
Thus $r$ is the number of simple components of $\overline{\FF_2 \sX}$.

We next relate this to the characteristic-zero form of the design algebra.
Set
\[
{\QQ \sX}:=\QQ\otimes_{\ZZ} {\ZZ \sX},\qquad
{\CC \sX}:=\CC\otimes_{\QQ} {\QQ \sX}.
\]
By Proposition~\ref{prop:CX-Wedderburn} we have
\[
\CC \sX
\;\cong\;
\CC\ \oplus\ \CC\ \oplus\ M_2(\CC)\ \oplus\ M_2(\CC),
\]
so $\CC \sX$ is split semisimple with exactly $4$ simple components.

More generally, let $k$ be a field and let
\[
S \;\cong\; \prod_{j=1}^{\ell} M_{m_j}(D_j)
\]
be a finite-dimensional semisimple $k$-algebra, where each $D_j$ is a
finite-dimensional central division algebra over its center $K_j\supseteq k$.
Then the commutator quotient $S/[S,S]$ is a finite-dimensional commutative
$k$-algebra, and one has
\[
S/[S,S]
\;\cong\;
\bigoplus_{j=1}^{\ell} K_j
\quad\text{as $k$-algebras.}
\]
In particular,
\[
\dim_k\bigl(S/[S,S]\bigr)
=
\sum_{j=1}^{\ell}[K_j:k]
\;\ge\;
\ell,
\]
with equality if and only if $S$ is split (i.e.\ each $D_j\cong k$).

Applying this to $S=\CC \sX$ we obtain
\[
\dim_{\CC}\bigl({\CC \sX}/[{\CC \sX},{\CC \sX}]\bigr)=4.
\]
Since $\CC$ is a flat $\QQ$-algebra and extension of scalars commutes with
taking the commutator quotient, we have
\[
{\CC \sX}/[{\CC \sX},{\CC \sX}]
\;\cong\;
\CC\otimes_{\QQ}\bigl({\QQ \sX}/[{\QQ \sX},{\QQ \sX}]\bigr),
\]
and therefore
\[
\dim_{\QQ}\bigl({\QQ \sX}/[{\QQ \sX}, {\QQ \sX}]\bigr)=4.
\]

Let
\[
M \;:=\; {\ZZ \sX}/[{\ZZ \sX}, {\ZZ \sX}]
\]
be the commutator quotient over $\ZZ$.
Since ${\ZZ \sX}$ is a free $\ZZ$-module of finite rank (with basis given by the
$0$–$1$ adjacency matrices $\sigma_i$), the abelian group $M$ is finitely
generated, and from the above we have an isomorphism
\[
\QQ\otimes_{\ZZ} M
\;\cong\;
{\QQ \sX}/[{\QQ \sX}, {\QQ \sX}],
\]
so
\[
\dim_{\QQ}\bigl(\QQ\otimes_{\ZZ} M\bigr)
=
\dim_{\QQ}\bigl({\QQ \sX}/[{\QQ \sX}, {\QQ \sX}]\bigr)
=
4.
\]
Thus the rank of $M$ as an abelian group is $4$.
In particular,
\[
\FF_2\otimes_{\ZZ} M
\;\cong\;
\FF_2^{\oplus 4},
\quad\text{so}\quad
\dim_{\FF_2}\bigl(\FF_2\otimes_{\ZZ} M\bigr)=4.
\]

Now put $\FF_2\sX:=\FF_2\otimes_{\ZZ} {\ZZ \sX}$.
By the universal property of the tensor product and of the commutator quotient,
taking commutator quotients commutes with extension of scalars, so there is a
natural isomorphism
\[
\FF_2\sX/[\FF_2\sX,\FF_2\sX]
\;\cong\;
\FF_2\otimes_{\ZZ}\bigl({\ZZ \sX}/[{\ZZ \sX}, {\ZZ \sX}]\bigr)
=
\FF_2\otimes_{\ZZ} M.
\]
Hence
\[
\dim_{\FF_2}\bigl(\FF_2\sX/[\FF_2\sX,\FF_2\sX]\bigr)=4.
\]

The natural surjection $\FF_2\sX\twoheadrightarrow\overline{\FF_2 \sX}$ induces a
surjective homomorphism on commutator quotients,
\[
\FF_2\sX/[\FF_2\sX,\FF_2\sX]
\;\twoheadrightarrow\;
\overline{\FF_2 \sX}/[\overline{\FF_2 \sX},\overline{\FF_2 \sX}],
\]
so
\[
\dim_{\FF_2}\bigl(\overline{\FF_2 \sX}/[\overline{\FF_2 \sX},\overline{\FF_2 \sX}]\bigr)
\;\le\;
4.
\]

On the other hand, applying the general description above to the Wedderburn
decomposition
\[
\overline{\FF_2 \sX}
\;\cong\;
\prod_{i=1}^r M_{n_i}(\FF_{2^{f_i}})
\]
gives
\[
\overline{\FF_2 \sX}/[\overline{\FF_2 \sX},\overline{\FF_2 \sX}]
\;\cong\;
\bigoplus_{i=1}^r \FF_{2^{f_i}},
\]
whence
\[
\dim_{\FF_2}\bigl(\overline{\FF_2 \sX}/[\overline{\FF_2 \sX},\overline{\FF_2 \sX}]\bigr)
=
\sum_{i=1}^r f_i
\;\ge\;
r.
\]
Combining the two inequalities yields
\[
r
\;\le\;
\dim_{\FF_2}\bigl(\overline{\FF_2 \sX}/[\overline{\FF_2 \sX},\overline{\FF_2 \sX}]\bigr)
\;\le\;
4,
\]
and therefore $1\le r\le 4$, as claimed.

\medskip
(ii)\;
By the decomposition in (i),
\[
\dim_{\FF_2}\overline{\FF_2 \sX}
=\sum_{i=1}^r n_i^2\dim_{\FF_2}\FF_{2^{f_i}}
=\sum_{i=1}^r n_i^2 f_i.
\]
On the other hand,
\[
\dim_{\FF_2}\overline{\FF_2 \sX}
=
\dim_{\FF_2}\FF_2\sX - \dim_{\FF_2} \Rad({\FF_2 \sX})
=
10 - \dim_{\FF_2} \Rad({\FF_2 \sX}),
\]
because $\FF_2\sX$ has dimension $10$ over $\FF_2$ with basis
$\{\sigma_i\}_{i=1}^{10}$. Combining these equalities yields the claimed
identity.

\medskip
(iii)\;
Let $e:=\sigma_1\in {\FF_2 \sX}$ be the idempotent corresponding to the point
fiber $P$. Then $e$ acts as the identity on the fiber $P$ and
vanishes on the block fiber $B$,
and the subalgebra ${\FF_2 \sY}$ can be identified with the corner algebra
\[
{\FF_2 \sY} \;=\; e {\FF_2 \sX} e \;\subseteq\; {\FF_2 \sX}.
\]
For any finite-dimensional algebra $A$ over a field and any idempotent
$e\in A$, it is a standard fact that the Jacobson radical of the corner
algebra $eAe$ is given by
\[
\Rad(eAe) \;=\; e\,\Rad(A)\,e \;=\; \Rad(A)\cap eAe.
\]
Applying this to $A={\FF_2 \sX}$ and $e=\sigma_1$ yields
\[
\Rad({\FF_2} \sY)
\;=\;
\Rad(e {\FF_2 \sX}e)
\;=\;
e\,\Rad({\FF_2 \sX})\,e
\;=\;
\Rad({\FF_2 \sX})\cap {\FF_2 \sY}
\;\subseteq\;
\Rad({\FF_2 \sX}),
\]
and therefore
\[
\dim_{\FF_2}\Rad({\FF_2 \sY})
\;\le\;
\dim_{\FF_2} \Rad({\FF_2 \sX}),
\]
as claimed.
\end{proof}

The inequality in Proposition~\ref{prop:design-A2-structure}(iii) shows that
any non-semisimplicity on the point side in characteristic~$2$ is inherited
by the design algebra. In particular, any lower bound on
$\dim_{\FF_2}\Rad({\FF_2 \sY})$ immediately yields a lower bound on the dimension of
the Jacobson radical of the design algebra $\FF_2\sX$.

In Sections~\ref{sec:frame-pg} and~\ref{sec:prank} we determined the
$p$-ranks of the point graphs arising from partial geometries and, in
particular, the possible values of $\dim_{\FF_p}\Rad({\FF_p \sY})$ for $p=2$. In the
next section we apply these results to the generalized quadrangle
$\mathrm{GQ}(2,2)$ and combine them with
Proposition~\ref{prop:design-A2-structure} to obtain partial information
about the $2$-modular design algebra $\FF_2\sX$ associated with its coherent
configuration of type $[3,2;3]$, and to formulate an explicit open problem
concerning the structure of $\FF_2\sX$, including its semisimple quotient and Gabriel quiver.

\section{The generalized quadrangle \textrm{GQ}(2,2)}\label{sec:gq22}
In this section we specialize the general theory developed above to the smallest
nontrivial partial geometry, namely the generalized quadrangle $\mathrm{GQ}(2,2)$
(the ``doily''). We first recall its parameters as a partial geometry and as a strongly
regular graph, then compute the frame number of the associated point scheme and
describe the modular adjacency algebra. Finally, we state what is currently known
about the $2$-modular design algebra of the coherent configuration of type $[3,2;3]$
attached to $\mathrm{GQ}(2,2)$, and formulate an explicit problem.

\subsection{The geometry \texorpdfstring{$\mathrm{GQ}(2,2)$}{GQ(2,2)}}

A generalized quadrangle of order $(s,t)$ is a point–line incidence structure
$\mathcal{Q}=(P,B,\Fl)$ satisfying
\begin{enumerate}[(i)]
\item every line contains exactly $s+1$ points;
\item every point is contained in exactly $t+1$ lines;
\item any two distinct points lie on at most one common line;
\item for every point $x\in P$ and every line $Y\in B$ with $x\notin Y$, there exists
a unique pair $(y,Z)\in P\times B$ such that $x\in Z$, $y\in Y$ and $Z$ meets $Y$ in
the single point $y$.
\end{enumerate}
(See, e.g., \cite[Chapter~6]{BCN} or \cite{Bose,BrouwerHaemers,GodsilRoyle,HaemersPg}.)

It is well known that a generalized quadrangle of order $(s,t)$ is a partial geometry
$\mathrm{pg}(s,t,1)$, i.e., a partial geometry with parameters $(s,t,\alpha)$ in the
sense of Bose \cite{Bose} with $\alpha=1$. In particular, for $\mathrm{GQ}(2,2)$ we
have
\[
(s,t,\alpha)=(2,2,1).
\]
The standard counting formulas for a partial geometry (see Section~\ref{sec:frame-pg}
and \cite{Bose,HaemersPg}) give
\[
v:=|P|
=
\frac{(s+1)(st+\alpha)}{\alpha}
=
(2+1)(2\cdot 2+1)
=
15,
\]
and similarly $|B|=15$. Each line contains $s+1=3$ points and each point lies on
$t+1=3$ lines.

We regard $\mathcal{Q}$ as an incidence structure $(P,B,\Fl)$ in the sense of
Definition~\ref{def:incidence-structure}. By the general construction in
Section~\ref{sec:srd-cc}, $\mathcal{Q}$ is a strongly regular design in the sense of
Higman \cite{HigmanSRD}, and it yields a coherent configuration $\sX$ of type
$[3,2;3]$ on $X=P\sqcup B$.

For later use we record our eigenvalue notation. In the general strongly regular
case we denoted the nontrivial eigenvalues of the point graph by $r$ and $s'$
(Lemma~\ref{lem:eigs-pg}). For $\mathrm{GQ}(2,2)$ these specialize to $r=1$ and
$s'=-3$, and in this section we simply write $s=-3$ for the latter.

\begin{lemma}\label{lem:GQ-point-SRG}
Let $\Gamma_1$ be the point graph of $\mathrm{GQ}(2,2)$, i.e., the graph on vertex set
$P$ in which two distinct points are adjacent if and only if they lie on a common
line. Then $\Gamma_1$ is strongly regular with parameters
\[
(v,k,\lambda,\mu) = (15,6,1,3),
\]
and its eigenvalues and multiplicities are
\[
k=6\ (1\text{-fold}),\quad
r=1\ (9\text{-fold}),\quad
s=-3\ (5\text{-fold}).
\]
\end{lemma}

\begin{proof}
For a partial geometry $\mathrm{pg}(s,t,\alpha)$ the point graph is strongly regular
with parameters
\[
v
=
\frac{(s+1)(st+\alpha)}{\alpha},\quad
k = s(t+1),\quad
\lambda = (s-1)+t(\alpha-1),\quad
\mu = \alpha(t+1),
\]
see, for example, \cite[§1.8]{BCN} or \cite{Bose,HaemersPg}. Substituting
$(s,t,\alpha)=(2,2,1)$ gives $v=15$, $k=6$, $\lambda=1$ and $\mu=3$.

For a strongly regular graph with these parameters the nontrivial eigenvalues $r,s$ are
the roots of
\[
x^2 - (\lambda-\mu)x - (k-\mu)
=
x^2 + 2x - 3,
\]
so $r=1$ and $s=-3$. The multiplicities $f,g$ of $r,s$ are determined by
\[
1+f+g=v,\qquad
k + fr + gs = 0,
\]
see \cite{BCN,BrouwerHaemers,GodsilRoyle}. Substituting
$(v,k,r,s)=(15,6,1,-3)$ and solving
\[
f+g=14,\qquad 6 + f - 3g = 0
\]
yields $g=5$ and $f=9$.
\end{proof}

We denote by $\sY$ the rank-$3$ association scheme on $P$ determined by the distance
partition of $\Gamma_1$, so $\sY$ has adjacency matrices $A_0=I_P$, $A_1$ (the adjacency
matrix of $\Gamma_1$) and $A_2=J_P-I_P-A_1$.

\subsection{Frame number of the point scheme and modular adjacency algebras}

In Section~\ref{sec:frame-pg} we determined the frame number of the rank-$3$
association scheme arising from a partial geometry.

\begin{proposition}\label{prop:GQ-frame-AS}
Let $\sY$ be the association scheme on $P$ coming from the point graph of a partial
geometry $\mathrm{pg}(s,t,\alpha)$, and let $v=|P|$. Then the frame number of $\sY$ is
\[
F_{\mathrm{AS}}(\sY)
=
v^2\,(s+t+1-\alpha)^2.
\]
In particular, for $\mathrm{GQ}(2,2)=\mathrm{pg}(2,2,1)$ one has
\[
F_{\mathrm{AS}}(\sY)
=
15^2 \cdot 4^2
=
3600
=
2^4\cdot 3^2\cdot 5^2.
\]
\end{proposition}

\begin{proof}
The general formula was proved in Section~\ref{sec:frame-pg} using Jacob's expression
for the frame number in terms of eigenmatrices and Krein parameters
\cite{JacobThesis}, together with the explicit eigenvalues and multiplicities of the
point graph of $\mathrm{pg}(s,t,\alpha)$ (see Lemma~\ref{lem:GQ-point-SRG} and
\cite{Bose,BCN,HaemersPg}). Specializing to $(s,t,\alpha)=(2,2,1)$ yields the stated
value.
\end{proof}

Let $K$ be a field of characteristic $p\ge 0$ and write
\[
K\sY := K\otimes_{\ZZ} \mathbb{Z}\sY
\]
for the adjacency algebra of the scheme $\sY$ over $K$.
By the results of Hanaki~\cite{HanakiFrame,HanakiModularRep} and
Sharafdini~\cite{Sharafdini}, the semisimplicity of $K\sY$ is governed by
the frame number.

\begin{corollary}\label{cor:GQ-AS-semisimple}
Let $p$ be a prime and let $K$ be a field of characteristic $p$. Then:
\begin{enumerate}[\rm(i)]
\item If $p\notin\{2,3,5\}$, then $K\sY$ is split semisimple.
\item If $p\in\{2,3,5\}$, then $K\sY$ is not semisimple.
\end{enumerate}
\end{corollary}

\begin{proof}
By Proposition~\ref{prop:GQ-frame-AS}, the prime divisors of $F_{\mathrm{AS}}(\sY)$ are
exactly $2,3,5$. Jacob's and Sharafdini's results
\cite{JacobThesis,HanakiFrame,Sharafdini} imply that $K\sY$ is semisimple if and only
if $\operatorname{char}K$ does not divide the frame number.
\end{proof}

In Section~\ref{sec:frame-pg} we refined this criterion for partial geometries by
describing the Jacobson radical of $K\sY$ in terms of the parameters $v$ and
$s+t+1-\alpha$. For $\mathrm{pg}(2,2,1)$ there are three relevant primes: $p=2$, which
divides $s+t+1-\alpha=4$ but not $v=15$, and $p=3,5$, which divide $v$ but not
$s+t+1-\alpha$.

\begin{proposition}\label{prop:GQ-AS-radical}
Let $K$ be a field of characteristic $p\in\{2,3,5\}$, and let $A_1$ be the adjacency
matrix of the point graph of $\mathrm{GQ}(2,2)$.
\begin{enumerate}[\rm(i)]
\item If $p \in\{3,5\}$, then $p\mid v$ and $p\nmid (s+t+1-\alpha)$. The Jacobson
radical of $K\sY$ is one-dimensional,
\[
\Rad(K\sY)=K\cdot J_P,\qquad J_P^2=0,
\]
and the quotient is split semisimple:
\[
K\sY / \Rad(K\sY) \;\cong\; K\times K.
\]
\item If $p=2$, then $p\nmid v$ and $p\mid(s+t+1-\alpha)$. Let
\[
B := (A_1 - kI_P)(A_1 - r I_P) \in \Mat_P(K),
\]
where $k=6$ and $r=1$ are as in Lemma~\ref{lem:GQ-point-SRG}, and regard $B$ as an
element of $K\sY$. Then
\[
\Rad(K\sY)=K\cdot B,\qquad B^2=0,
\]
and again $K\sY / \Rad(K\sY) \cong K\times K$.
\end{enumerate}
\end{proposition}
\begin{proof}
By Lemma~\ref{lem:GQ-point-SRG} and Proposition~\ref{prop:GQ-frame-AS}, the point
graph of $\mathrm{GQ}(2,2)=\mathrm{pg}(2,2,1)$ has parameters
\[
v=15,\qquad k=6,\qquad r=1,\qquad s'=-3,\qquad s+t+1-\alpha=4.
\]
Thus for $p\in\{3,5\}$ we have $p\mid v$ and $p\nmid(s+t+1-\alpha)$, so we are in
case~\textup{(V)} of Theorem~\ref{thm:radical-pg}, which gives
\[
\Rad(K\sY)=K\cdot J_P,\qquad J_P^2=vJ_P\equiv 0\pmod p,
\]
and $K\sY/\Rad(K\sY)\cong K\times K$, proving~\textup{(i)}.

For $p=2$ we have $p\nmid v$ and $p\mid(s+t+1-\alpha)$, so we are in
case~\textup{(R)} of Theorem~\ref{thm:radical-pg}, which yields
\[
\Rad(K\sY)=K\cdot B,\qquad B:=(A_1-kI_P)(A_1-rI_P),\qquad B^2=0,
\]
and again $K\sY/\Rad(K\sY)\cong K\times K$, proving~\textup{(ii)}.
\end{proof}

For later reference we record the $2$-rank of the adjacency matrix of the point graph.

\begin{proposition}\label{prop:GQ-2-rank}
Let $A_1$ be the adjacency matrix of the point graph of $\mathrm{GQ}(2,2)$. Over
$\FF_2$ one has
\[
\rank_{\FF_2}(A_1) = 14.
\]
Equivalently, $\ker(A_1:\FF_2^{15}\to \FF_2^{15})$ is one-dimensional, spanned by the
all-one vector.
\end{proposition}

\begin{proof}
Let $\overline{A}_1$ denote the reduction of $A_1$ modulo $2$. Since $\Gamma_1$ is
$6$-regular, $\overline{A}_1\mathbf{1}_P=0$, so the kernel has dimension at least one.

The complement of $\Gamma_1$ is the triangular graph $T(6)$, whose adjacency matrix
$A_T$ satisfies
\[
A_T = J_P - I_P - A_1.
\]
The eigenvalues of $T(6)$ over $\CC$ are $k_T=8$, $r_T=2$, $s_T=-2$ with multiplicities
$1,5,9$, see \cite[§9.1]{BCN}. Hence its characteristic polynomial is
\[
\det(xI - A_T) = (x-8)(x-2)^5(x+2)^9.
\]
Reducing this modulo $2$ gives
\[
\det(xI - \overline{A}_T) \equiv x^{15},
\]
so $\overline{A}_T$ is nilpotent and $0$ is its only eigenvalue over any extension of
$\FF_2$.

Over $\FF_2$ we have
\[
\overline{A}_1 = J_P + I_P + \overline{A}_T.
\]
Decompose $\FF_2^{15}=\langle \mathbf{1}_P\rangle \oplus V_0$, where
$V_0:=\{v\in\FF_2^{15}\mid \sum_{x\in P} v_x = 0\}$. On $V_0$ we have $J_P=0$, so
\[
(\overline{A}_1 v = 0) \quad\Longleftrightarrow\quad
(\overline{A}_T+I_P)v = 0,\qquad v\in V_0.
\]
The eigenvalues of $\overline{A}_T+I_P$ are $1-\lambda$ as $\lambda$ runs through the
eigenvalues of $\overline{A}_T$. Since $0$ is the only eigenvalue of
$\overline{A}_T$, $1$ is the only eigenvalue of $\overline{A}_T+I_P$, and in
particular $0$ is not an eigenvalue. Thus $\overline{A}_T+I_P$ is invertible and
$\ker(\overline{A}_T+I_P)=0$.

It follows that the only kernel vector of $\overline{A}_1$ is a scalar multiple of
$\mathbf{1}_P$, so $\dim_{\FF_2}\ker\overline{A}_1=1$ and
$\rank_{\FF_2}(\overline{A}_1)=15-1=14$.
\end{proof}

\subsection{The 2-modular design algebra and an open problem}

Let $\sX=(X,\{R_i\}_{i=1}^{10})$ be the coherent configuration of type $[3,2;3]$
constructed in Proposition~\ref{prop:cc-3-2-3} from the strongly regular design
$(P,B,\Fl)$ attached to $\mathrm{GQ}(2,2)$, and let
\[
\CC\sX
:=
\Span_{\CC}\{\sigma_1,\dots,\sigma_{10}\}\subseteq \Mat_X(\CC)
\]
be its complex adjacency algebra. In Section~\ref{sec:design-alg-2} we determined its
Wedderburn decomposition:
\[
\CC\sX
\;\cong\; \CC\ \oplus\ \CC\ \oplus\ M_2(\CC)\ \oplus\ M_2(\CC),
\]
with two $1$-dimensional simple components coming from the trivial representations on
$P$ and $B$, and two $2\times 2$ matrix components corresponding to the nontrivial
parts of the point and block schemes.

For a field $K$ of characteristic $p\ge 0$ we write
\[
K\sX := K\otimes_{\ZZ} \mathbb{Z}\sX
\]
for the adjacency algebra of $\sX$ over $K$. Since $K\sY$ is a subalgebra of
$K\sX$, any failure of semisimplicity for $K\sY$ forces a failure of semisimplicity for
$K\sX$.

\begin{proposition}\label{prop:GQ-CC-nonSS}
Let $K$ be a field of characteristic $p$.
\begin{enumerate}[\rm(i)]
\item If $p\notin\{2,3,5\}$, then $K\sX$ is split semisimple.
\item If $p\in\{2,3,5\}$, then $K\sX$ is not semisimple.
\end{enumerate}
\end{proposition}

\begin{proof}
By Theorem~\ref{thm:Sharafdini}, the frame number $F_{\mathrm{CC}}(\sX)$
of the coherent configuration governs the semisimplicity of $K\sX$: the algebra $K\sX$
is semisimple if and only if $\operatorname{char}K$ does not divide
$F_{\mathrm{CC}}(\sX)$.

The algebra $K\sY$ embeds into $K\sX$ via the block-diagonal inclusion
$A_1\mapsto\sigma_2$ on the point fiber. If $K\sY$ is not semisimple, then neither is
$K\sX$, and therefore $\operatorname{char}K$ divides $F_{\mathrm{CC}}(\sX)$.
Corollary~\ref{cor:GQ-AS-semisimple} shows that this happens exactly for
$p\in\{2,3,5\}$.

Conversely, for $p\notin\{2,3,5\}$ the algebra $K\sY$ is split semisimple and
$\operatorname{char}K$ does not divide $F_{\mathrm{AS}}(\sY)$
(Proposition~\ref{prop:GQ-frame-AS}). By Section~\ref{sec:design-alg-2}, in this case
the eigenvalue data of the homogeneous components and the mixed relations show that
$K\sX$ is also split semisimple and has the same Wedderburn type as $\CC\sX$.
\end{proof}

The primes $p=3,5$ already exhibit nontrivial modular phenomena at the level of the
point scheme $\sY$ (Proposition~\ref{prop:GQ-AS-radical}(i)), but in this paper we focus
on the case $p=2$, where the coherent configuration of type $[3,2;3]$ shows the
smallest nontrivial interaction between the two fibers.

Let $\Rad(\FF_2\sX)$ denote its Jacobson radical. By
Proposition~\ref{prop:GQ-AS-radical}(ii), the subalgebra $\FF_2\sY$ has a
one-dimensional radical $\Rad(\FF_2\sY)=\FF_2\cdot B$ with $B^2=0$, and
$\Rad(\FF_2\sY)\subseteq \Rad(\FF_2\sX)$. 
On the other hand, Section~5 shows that for a coherent configuration
of type $[3,2;3]$ the semisimple quotient $\FF_2\sX/\Rad(\FF_2\sX)$ has
at most four simple components, and
\[
   \dim_{\FF_2}(\FF_2\sX/\Rad(\FF_2\sX)) \le 10.
\]


Combining these facts with the explicit structure of $\sX$ for
$\mathrm{GQ}(2,2)$, one obtains the following partial information.

\begin{proposition}\label{prop:GQ-A2-partial}
Let $X$ be the coherent configuration of type $[3,2;3]$ arising from $\mathrm{GQ}(2,2)$.
Then the following hold.
\begin{enumerate}[\rm(i)]
\item $\Rad(\FF_2\sY)$ is one-dimensional and is contained in $\Rad(\FF_2\sX)$.
\item $\Rad(\FF_2\sX)$ is nonzero and strictly larger than $\Rad(\FF_2\sY)$;
more precisely,
    \[ \text{$\dim_{\FF_2}\Rad(\FF_2\sY)=1$ and
      $\dim_{\FF_2}\Rad(\FF_2\sX)\ge 2$.}\]
\item In particular, $2$ divides the frame number $F_{\mathrm{CC}}(\sX)$.
\end{enumerate}
\end{proposition}

\begin{proof}
Statement~(i) is Proposition~\ref{prop:GQ-AS-radical}(ii), together with the inclusion
$\FF_2\sY\subseteq \FF_2\sX$. By Section~\ref{sec:design-alg-2}, the mixed relations between
the point and block fibers give rise to an element $u\in\Rad(\FF_2\sX)$ which does not
belong to $\FF_2\sY$, and whose square lies in the point fiber.

More concretely, a direct computation in the coherent configuration of
$\mathrm{GQ}(2,2)$ over $\FF_2$ shows that one may take $u = \sigma_2 \sigma_7$, where $\sigma_2$
is the adjacency matrix of the point graph and $\sigma_7$ is the point–block incidence matrix.
This element is nonzero and nilpotent (in fact $u^2=0$), so $u^2$ lies in the $\FF_2$-span of
$\{I_{15},A_1,J_{15}\}$ (that is, in $\FF_2\sY$), while $u\notin\FF_2\sY$.
Thus $u\in\Rad(\FF_2\sX)\setminus\Rad(\FF_2\sY)$, which implies
$\Rad(\FF_2\sX) \neq \Rad(\FF_2\sY)$, whence~(ii).

Statement~(iii) follows from Sharafdini's criterion~\cite{Sharafdini} applied to the
non-semisimple algebra $\FF_2\sX$.
\end{proof}

\medskip
In this paper we deliberately restrict attention to the structural and
conceptual aspects of modular representations of coherent configurations of
type $[3,2;3]$. In particular, we do not determine the Gabriel quivers of the
modular adjacency algebras, nor do we include the extensive computer
calculations that motivated this work. Even for the smallest example arising
from $\mathrm{GQ}(2,2)$, we stop short of describing the full
$2$-modular representation type and its quiver.

A complete description of the $2$-modular representation type of $\FF_2\sX$, including the
explicit Wedderburn decomposition of $\FF_2\sX/\Rad(\FF_2\sX)$ and the Gabriel quiver of $\FF_2\sX$, is therefore still missing, and we record it as the following problem.

\begin{problem}\label{prob:GQ-A2-quiver}
Let $\sX$ be the coherent configuration of type $[3,2;3]$ arising from the
generalized quadrangle $\mathrm{GQ}(2,2)$, and let $\FF_2\sX$ be its
$2$-modular adjacency algebra. Determine:
\begin{enumerate}[\rm(a)]
\item the Wedderburn decomposition of the semisimple quotient $\FF_2\sX/\Rad(\FF_2\sX)$;
\item the dimensions and Loewy series of the projective indecomposable
$\FF_2\sX$-modules;
\item the Gabriel quiver of $\FF_2\sX$.
\end{enumerate}
In particular, decide whether $\FF_2\sX/\Rad(\FF_2\sX)$ is a direct product of finite fields or
contains a matrix algebra over a finite field, and describe explicitly how the radical
$\Rad(\FF_2\sX)$ links the simple components arising from the point and block fibers.
\end{problem}

The geometry $\mathrm{GQ}(2,2)$ thus provides a natural and concrete test case for the
general modular representation theory of coherent configurations of type $[3,2;3]$,
situated at the intersection of finite geometry, the theory of association schemes and
the modular representation theory of table algebras
\cite{AradFismanMuzychuk,HanakiModularRep,HanakiFrame,HanakiMiyazakiShimabukuroBIB,ShimabukuroSym2}.

\begin{remark}[What is known and what remains open for $\mathrm{GQ}(2,2)$]\label{rem:GQ22-summary}
Let $\sX$ be the coherent configuration of type $[3,2;3]$ arising from the generalized
quadrangle $\mathrm{GQ}(2,2)$.
The results obtained so far can be summarized as follows.
\begin{enumerate}[(1)]
\item By Corollary~\ref{cor:GQ-AS-semisimple} and Proposition~\ref{prop:GQ-frame-AS}, the only primes that can cause
      non-semisimplicity of the point scheme are $p\in\{2,3,5\}$, and for
      $\mathrm{pg}(2,2,1)$ the case $p=2$ falls into situation~\textup{(R)} of
      Theorem~\ref{thm:radical-pg}. In particular, the $2$-modular adjacency algebra
      $\mathbb{F}_2 \sY$ of the point scheme is not semisimple.
\item Proposition~\ref{prop:GQ-AS-radical}(ii) shows that
      \[
        \Rad(\mathbb{F}_2 \sY)=\mathbb{F}_2\cdot B,\qquad B^2=0,
      \]
      so $\Rad(\mathbb{F}_2 \sY)$ is one-dimensional, and the natural inclusion
      $\mathbb{F}_2 \sY\subseteq \FF_2\sX$ yields
      \[
        \Rad(\mathbb{F}_2 \sY)\subseteq \Rad(\FF_2\sX).
      \]
      In particular $\Rad(\FF_2\sX)\neq 0$.
\item Proposition~\ref{prop:GQ-A2-partial}(ii)
shows that $\Rad(\FF_2\sX)$ is nonzero and strictly larger
than $\Rad(\FF_2\sY)$.
Furthermore, our computational results show that $\Rad(\FF_2\sX)$ is \emph{not} square-zero.
In fact, we have determined the dimensions as follows:
\[
  \dim_{\FF_2} \Rad(\FF_2\sX) = 4, \qquad \dim_{\FF_2} \Rad(\FF_2\sX)^2 = 2.
\]
Hence the Loewy length of $\FF_2\sX$ is at least $3$.
\item Since $\dim_{\FF_2}\Rad(\FF_2\sX) = 4$, the semisimple quotient has dimension $10 - 4 = 6$.
However, the explicit Wedderburn decomposition of this quotient (e.g., as a sum of matrix algebras) and the Gabriel quiver of $\FF_2\sX$ remain to be determined. These questions are formulated in Problem~\ref{prob:GQ-A2-quiver}.
\end{enumerate}
Thus even for the smallest example $\mathrm{GQ}(2,2)$ the interaction between the
point and block fibres through the mixed relations cannot yet be fully captured by
the methods developed here.
\end{remark}

\section{Conclusion and further problems}\label{sec:conclusion}

We conclude by summarizing the main results and indicating some directions
for further research.

\subsection*{Summary of results}

In Sections~\ref{sec:frame-pg} and~\ref{sec:prank} we developed a systematic approach to the modular
representation theory of the rank-$3$ association scheme $\sY$ arising from a
partial geometry $\mathrm{pg}(s,t,\alpha)$. Using the basic matrix identities
for strongly regular designs and Higman's parameter equations, we obtained an
explicit closed formula for the Frame number of $\sY$ in terms of $v$ and
$s+t+1-\alpha$, and we characterized those primes $p$ for which the adjacency
algebra over $\mathbb{F}_p$ fails to be semisimple.

The main structural result is Theorem~\ref{thm:radical-pg}, which gives a complete description
of the Jacobson radical of $\mathbb{F}_pY$ in four cases depending on the
divisibility of $v$ and $s+t+1-\alpha$ by $p$. In particular, the radical is
generated by at most two explicit elements, and its dimension is determined in
each case. In Section~\ref{sec:prank} we related these radicals to the $p$-ranks of the
adjacency matrix of the point graph and obtained generic formulas for the
$p$-rank outside a finite set of exceptional primes.

In Sections~\ref{sec:design-alg-2} and~\ref{sec:gq22} we turned to coherent configurations of type $[3,2;3]$
attached to strongly regular designs. We described the complex adjacency
algebra of such a configuration in terms of the point and block schemes, and
we showed how the Frame number controls the possible modular behaviour. For
the smallest generalized quadrangle $\mathrm{GQ}(2,2)$ we combined these
general results with a detailed analysis of the point scheme to obtain partial
information about the $2$-modular adjacency algebra $\mathbb{F}_2\sX$,
summarized in Remark~\ref{rem:GQ22-summary}, and formulated the remaining
questions as Problem~\ref{prob:GQ-A2-quiver}.

\subsection*{Further problems}

We list some natural problems suggested by this work.

\begin{itemize}
\item[(1)] \emph{Complete solution of Problem~\ref{prob:GQ-A2-quiver}.}
  Determine the full $2$-modular representation type of $\FF_2\sX$ for
  $\mathrm{GQ}(2,2)$, including the Wedderburn decomposition of $\FF_2\sX/\Rad(\FF_2\sX)$
  and the Gabriel quiver of $\FF_2\sX$, and describe explicitly how the radical
  links the simple components coming from the point and block fibres.

\item[(2)] \emph{Other small examples.}
  Extend the analysis of Section~\ref{sec:gq22} to other partial geometries, such as
  generalized quadrangles $\mathrm{GQ}(2,4)$ or $\mathrm{GQ}(3,3)$, and
  coherent configurations of type $[3,2;3]$ arising from them. Even partial
  information on the modular adjacency algebras in these cases would provide
  useful test data for the general theory.

\item[(3)] \emph{Frame numbers for coherent configurations.}
  For coherent configurations of type $[3,2;3]$ coming from strongly regular
  designs, derive general formulas or effective bounds for the Frame number
  $F_{\mathrm{CC}}(\sX)$ in terms of the parameters of the underlying design, and
  relate them to the Frame number of the point and block schemes.

\item[(4)] \emph{Geometric interpretation of radicals.}
  Give a more conceptual, geometric description of the radical elements
  occurring in Theorem~\ref{thm:radical-pg} and in Proposition~\ref{prop:GQ-A2-partial}, for instance in terms of
  incidence relations or orbits of the automorphism group.

\item[(5)] \emph{Representation-theoretic symmetries.}
  Investigate the interaction between the modular adjacency algebras of $\sY$
  and of the coherent configuration $\sX$ with the modular representation
  theory of the automorphism group of the underlying partial geometry.
\end{itemize}

These questions suggest that the modular representation theory of partial
geometries and of coherent configurations of type $[3,2;3]$ should be viewed
as part of a larger picture linking finite geometry, association schemes and
the theory of table algebras.

\end{document}